\documentclass[10pt, reqno]{amsart}

\usepackage{amsmath}
\usepackage{amsfonts}
\usepackage{amssymb}
\usepackage{latexsym}
\usepackage{bbm,amsmath,amsthm,epsfig,latexsym,marvosym}
\usepackage{amsfonts}

\newtheorem{theorem}{Theorem}[section]
\newtheorem{lemma}[theorem]{Lemma}
\newtheorem{proposition}[theorem]{Proposition}
\newtheorem{corollary}[theorem]{Corollary}
\newtheorem{remark}[theorem]{Remark}
\newtheorem{definition}[theorem]{Definition}
\newtheorem{example}[theorem]{Example}

\newcommand{\PP}{{\Phi}}
\newcommand{\g}{\gamma}
\newcommand{\e}{\epsilon}
\newcommand{\DIV}{{\rm div \,}}

\newcommand{\tr}{{\rm tr}}
\newcommand{\bary}{{\rm bar }\,}
\newcommand{\Rn}{\mathbb{R}^n}
\newcommand{\pa}{\partial}
\newcommand{\BPP}{\beta_{\PP}}
\newcommand{\Hn}{\mathcal{H}^{n-1}}
\newcommand{\M}{M_{\PP}}
\newcommand{\m}{m_{\PP}}

\newcommand{\NN}{\mathcal{N}}

\numberwithin{equation}{section}
\begin{document}

\title[A strong form of the quantitative Wulff inequality]{A strong form of the quantitative Wulff inequality}

\author{Robin Neumayer}
\address{Department of Mathematics, University of Texas at Austin, Austin, TX, USA}
\email{rneumayer@math.utexas.edu}

\begin{abstract}
Quantitative isoperimetric inequalities are shown for anisotropic surface energies where
  the isoperimetric deficit controls both the Fraenkel asymmetry and a measure of the oscillation 
  of the boundary with respect to the boundary of the corresponding Wulff shape. 
  \end{abstract}

\maketitle

\section{Introduction}
\subsection{The Wulff inequality and stability}
For $n\geq 2$, the well-known isoperimetric inequality states that 
$$P(E) \geq n|B|^{1/n}|E|^{1/n'},$$
$n' = n/(n-1)$, with equality if and only if $E$ is a translation or dilation of $B=\{ x \in \Rn : |x| <1\}$, the Euclidean unit ball in $\Rn$. The isoperimetric inequality holds for all sets of finite perimeter $E \subset \Rn$, with the perimeter of $E$ equal to
$ P(E) = \Hn(\pa^* E). $
Here, $\pa^*E$ is the reduced boundary of $E$; see Section~\ref{SOFP}, or \cite{maggi2012sets} for a more complete overview.

The anisotropic surface energy is a natural generalization of the notion of perimeter and has applications in modeling of equilibrium configurations for solid crystals (see \cite{Wulff1901, Herring1951, taylor1978}) and of phase transitions (see \cite{Gurtin1985}).
We introduce a \textit{surface tension} $f: \Rn \to [0, + \infty)$ to be a convex positively $1$-homogeneous function that is positive on $S^{n-1}.$
The corresponding (anisotropic) \textit{surface energy} of a set of finite perimeter $E \subset \Rn$ is defined by
$$\PP(E) = \int_{\partial^* E} f(\nu_E(x)) \,d\Hn(x).$$
Here, $\nu_E$ is the measure theoretic outer unit normal; see Section~\ref{SOFP}.
Just as the ball minimizes perimeter among sets of the same volume, as expressed by the isoperimetric inequality, the surface energy is uniquely minimized among sets of a given volume by translations and dilations of a fixed convex set $K$ determined by the surface tension. This set $K$ is given by
$$K = \underset{\nu \in S^{n-1}}\bigcap \{ x \in \Rn : x\cdot \nu <f(\nu)\} ,$$ 
and is known as the \textit{Wulff shape} of $\PP$. The minimality of the Wulff shape is expressed by the \textit{Wulff inequality}:
$$\PP(E) \geq \PP(K)\bigg(\frac{|E|}{|K|}\bigg)^{1/n'}= n |K|^{1/n}|E|^{1/n'},$$
with equality if and only if $E$ is a translation or dilation of $K$; see \cite{ taylor1978, milman1986asymptotic, fonseca1991wulff, fonseca1991uniqueness, dacorogna1992wulff, DGS92, brothers1994}.
In the case where the surface tension $f$ is constantly equal to one, $\PP(E)$ reduces to the perimeter $P(E)$ and the Wulff inequality reduces to the isoperimetric inequality.

Given a surface tension $f$, one may define the \textit{gauge function} $f_*:\Rn \to [0,+\infty)$
by 
$$f_*(x) = \sup\{ x \cdot \nu : f(\nu) \leq 1\}.$$  The gauge function provides another characterization of the Wulff shape:  $K = \{ x : f_*(x) <1\}.$

To quantify how far a set is from achieving equality in the Wulff inequality, we introduce the {\it anisotropic isoperimetric deficit}, or simply the {\it deficit}, of a set $E$, defined by
$$\delta_{\PP}(E) = \frac{\PP(E) }{n|K|^{1/n}|E|^{1/n'}} -1.$$
The deficit equals zero if and only if,  up to a set of measure zero, $E= x +rK$ for some $x\in\Rn$ and $r>0$. This quantity is invariant under translations and dilations of $E$, 
as well as modifications of $E$ by sets of measure zero. 

The optimality of the Wulff shape in the Wulff inequality naturally gives rise to the question of stability: does the deficit control the distance of a set from the Wulff shape? In other words, given a distance $d$ from the family $\{x+ rK : x\in \Rn, r>0 \}$, one wants to find inequalities of the form
\begin{equation}\label{stability}
\delta_{\PP}(E) \geq \omega( d(E)) ,
\end{equation}
where $\omega $ is a (possibly explicit) function such that $\omega(d(E)) \to 0^+$ as $d(E) \to 0^+$. Ideally, one hopes to find the function $\omega$ that provides the \textit{sharp} rate of decay. 
Such an inequality can be viewed as a \textit{quantitative} form of the Wulff inequality: by rearranging the deficit, \eqref{stability} becomes
$$ \PP(E) \geq \PP(K) \left(\frac{|E|}{|K|}\right)^{1/n'} + \omega( d(E)) n|K|^{1/n}|E|^{1/n'}.$$
In this way, $\omega(d(E))$ serves as a remainder term in the Wulff inequality.

A well studied distance is the 
\textit{asymmetry index}, $\alpha_{\PP}(E)$, defined by 
\begin{equation}\label{asym}
\alpha_{\PP}(E) = \min_{y \in \Rn} \left\{ \frac{|E\Delta (rK+y)|}{|E|} :\ |rK| = |E|\right\},
\end{equation}
where $E\Delta F =(E\setminus F) \cup( F \setminus E)$ is the symmetric difference of $E$ and $F$.
For the case $f$ constantly equal to one, the asymmetry index is known as the \textit{Fraenkel asymmetry}. The quantitative isoperimetric inequality
 with the Fraenkel asymmetry was proven in sharp form by Fusco, Maggi, and Pratelli in \cite{FMP08}.  
Using symmetrization techniques, they showed that if $E$ is a set of finite perimeter with $0<|E|<\infty$, then
\begin{equation}\label{FMPStatement}
\alpha_1(E)^2 \leq C(n) \delta_1 (E).
\end{equation}
Here and in the future, we use the notation $\delta_1$ and $\alpha_1$ for the deficit and asymmetry index corresponding to the perimeter.

Before this full proof of \eqref{FMPStatement} was given, several partial results were shown; see \cite{fuglede1989, Hall1992, HHW}. Another proof of \eqref{FMPStatement} was given in \cite{CiLe12}, introducing a technique known as the selection principle, where a penalization technique and the regularity theory for almost-minimizers of perimeter reduce the problem to the case shown in \cite{fuglede1989}.

Stability of the Wulff inequality was first addressed in \cite{Esposito2005}, without the sharp exponent. Figalli, Maggi, and Pratelli later proved the sharp analogue of \eqref{FMPStatement} for the Wulff inequality in \cite{FiMP10}, using techniques from optimal transport
  and Gromov's proof of the Wulff inequality. They showed that there exists a constant $C(n)$, independent of $f$, such that
  \begin{equation}\label{FiMPStatement}
  \alpha_{\PP}(E)^2 \leq C(n) \delta_{\PP}(E)
  \end{equation}
for any set of finite perimeter $E$ with $0<|E| < \infty$. In both (\ref{FMPStatement}) and (\ref{FiMPStatement}), the power $2$ is sharp. The constant $C(n)$ in \eqref{FiMPStatement} is explicit, a feature that is not shared with any other proofs of \eqref{FMPStatement}, \eqref{FuscoJulin}, or any of the quantitative inequalities  obtained in this paper.

In \cite{fuscojulin11}, Fusco and Julin proved a strong form of the quantitative isoperimetric inequality, improving \eqref{FMPStatement} by showing
\begin{equation}\label{FuscoJulin} 
\alpha_1 (E)^2 + \beta_1(E)^2 \leq C(n) \delta_1(E)
\end{equation}
for any set of finite perimeter $E$ with $0<|E| < \infty$, where the \textit{oscillation index} $\beta_1(E)$ is defined by
\begin{align}
\beta_1(E)& =
 \min_{y\in \Rn}\bigg\{
\bigg(
 \frac{1}{2n|B|^{1/n}|E|^{1/n'} }
 \int_{\partial^* E } \Big|\nu_E(x) -\nu_{B_r(y)}\Big(y + r\frac{x-y}{|x-y|}\Big)\Big|^2\, d\Hn (x)\bigg)^{1/2} : 
|B_r| = |E|\bigg\} \nonumber\\
\label{betaB} 
&= \min_{y\in \Rn}
\left( \frac{1}{n|B|^{1/n}|E|^{1/n'} }
 \int_{\partial^* E } 
1 - \frac{x-y}{|x-y|} \cdot \nu_E(x) \, d\Hn (x)\right)^{1/2} .
\end{align}
While the asymmetry index $\alpha_1(E)$ is an $L^1$ distance between $E$ and $B$, the oscillation index $\beta_1(E)$ quantifies the oscillation of $\pa^*E$ with respect to $\pa B$. In \cite[Proposition 1.2]{fuscojulin11}, $\beta_1(E)$ is shown to control $\alpha_1(E)$; see Proposition~\ref{poincare} for the analogous statement in the anisotropic case.
Once again, the power $2$ in (\ref{FuscoJulin}) is sharp for both $\alpha_1(E)$ and $\beta_1(E)$. 

Since \cite{fuscojulin11}, analogous strong form quantitative inequalities have been studied in several settings:
in Gauss space \cite{Eldan2015, BarBraJul14},
on the sphere \cite{BDuzFus13}, and
in hyperbolic $n$-space \cite{BDS2015}.
\subsection{Statements of the main theorems}

The goal of this paper is to address the question of what form the inequality $(\ref{FuscoJulin})$ takes in the case of the anisotropic surface tension $\PP$. In other words, we prove a \textit{strong form} of the quantitative Wulff inequality, improving \eqref{FiMPStatement} by adding a term to the left hand side that quantifies the oscillation of $\pa^*E$ with respect to $\pa K$. We define the \textit{$\PP$-oscillation index} by 
\begin{equation}\label{bpp}
\BPP(E) = \min_{y\in \Rn} 
\left( \frac{1}{n|K|^{1/n}|E|^{1/n'}}
 \int_{\partial^* E } f(\nu_E(x)) -\nu_E(x) \cdot \frac{x-y}{f_*(x-y) }\, d\Hn (x)\right)^{1/2}.
\end{equation}
The following theorem is a strong form of the quantitative Wulff inequality that holds for an arbitrary surface energy. 
\begin{theorem}\label{th1} There exists a constant $C$ depending only on $n$ such that 
\begin{equation}\label{statement1}
\alpha_{\PP}(E)^2 + \BPP(E)^{4n/(n+1)}\leq C \delta_{\PP}(E)
\end{equation}
for every set of finite perimeter $E$ with $0<|E|<\infty$.\end{theorem}
As in (\ref{FiMPStatement}), the constant is independent of $f$. We expect that, as in (\ref{FuscoJulin}), the sharp exponent for $\BPP(E)$ in \eqref{statement1}  should be $2$. With additional assumptions on the surface tension $f$, we prove the stability inequality in sharp form for two special cases.
\begin{definition}\label{lellipt}
A surface tension $f$ is \textit{$\lambda$-elliptic}, $\lambda>0$, if $f\in C^2(\Rn \setminus \{0\})$ and 
$$(\nabla^2 f(\nu)\tau)\cdot \tau \geq \frac{\lambda}{|\nu| } \left| \tau - \Big(\tau\cdot \frac{\nu}{|\nu|}\Big)\frac{\nu}{|\nu|} \right|^2$$
for $\nu, \tau  \in \Rn$ with 
$\nu\neq 0$.
\end{definition}
This is a uniform ellipticity assumption for $\nabla^2 f(\nu)$ in the tangential directions to $\nu$. If $f$ is $\lambda$-elliptic, then the corresponding Wulff shape $K$ is of class $C^2 $ and uniformly convex (see \cite{schneider2013convex}, page $111$). When $\PP$ is a surface energy corresponding to a $\lambda$-elliptic surface tension, the following sharp result holds. The constant depends on $\m $ and $\M$, a pair of constants defined in \eqref{mM} that describe how much $f$ stretches and shrinks unit-length vectors.
\begin{theorem} \label{Smooth}
Suppose $f$ is a $\lambda$-elliptic surface tension with corresponding surface energy $\PP$. There exists a constant $C$ depending on $n, \lambda, \m/ \M,$ and $\|\nabla^2 f\|_{C^0(\pa K)}$ such that 
\begin{equation}\label{smooth12}\alpha_{\PP}(E)^2 +\BPP(E)^2 \leq C \delta_{\PP}(E)
\end{equation}
for any set of finite perimeter $E$ with $0<|E|<\infty$.
\end{theorem}

The second case where we obtain the strong form quantitative Wulff inequality with the sharp power is the case of a crystalline surface tension. 
\begin{definition} \label{crystals}
A surface tension $f$ is \textit{crystalline} if it is the maximum of finitely many linear functions, in other words, if there exists a finite set $\{x_j\}_{j=1}^N \subset \Rn \setminus \{0\}, N\in \mathbb{N}$, such that 
$$f(\nu) = \underset{1\leq j\leq N}{\max} \{ x_j \cdot \nu\}\ \ \mbox{ 	} \text{ for all } \nu \in S^{n-1}.$$
\end{definition}
If $f$ is a crystalline surface tension, then the corresponding Wulff shape $K$ is a convex polyhedron. 
In dimension two, when $f$ is a crystalline surface tension, we prove the following sharp quantitative Wulff inequality.

\begin{theorem}\label{dim2} 
Let $n=2$ and suppose $f$ is a crystalline surface tension with corresponding surface energy $\PP$. There exists a constant $C$ depending on $f$ such that 
$$\alpha_{\PP}(E)^2 +\BPP(E)^2 \leq C \delta_{\PP}(E)$$
for any set of finite perimeter $E$ with $0<|E|<\infty$.
\end{theorem}

Some remarks about the definition of the $\PP$-oscillation index $\beta_{\PP}$ in \eqref{bpp} are in order.
The oscillation index $\beta_1(E)$ in \eqref{betaB} measures oscillation of the reduced boundary of a set $E$ with respect to the boundary of the ball. Indeed, the quantity $\beta_1(E)$ is the integral over $\partial^*E$ of the Cauchy-Schwarz deficit $1 - \frac{x}{|x|}\cdot\nu_E(x)$, which quantifies in a Euclidean sense how closely $\nu_E(x)$ aligns with $\frac{x}{|x|}$.

To understand \eqref{bpp}, we remark that $f$ and $f_*$ are dual in the sense that they yield a Cauchy-Schwarz-type inequality called the Fenchel inequality, which states that 
$$\nu_E(x) \cdot \frac{x}{f_*(x)} \leq f(\nu_E(x)).$$ 
 Just as the oscillation index $\beta_1(E)$ quantifies the overall Cauchy-Schwarz deficit between $\frac{x}{|x|}$ and $\nu_E(x)$, the term $\BPP(E)$ is an integral along $\pa^*E$ of the  deficit in the Fenchel inequality. In Section~\ref{Subsection:anisotropies}, we show that $f(\nu_E(x)) = y \cdot \nu_E(x)$ for $y\in\partial K$ if and only if $y$ is a point on $\partial K$ where $\nu_E(x)$ is normal to a supporting hyperplane of $K$ at $y$. In this way, $\BPP(E)$ quantifies how much normal vectors of $E$ align with corresponding normal vectors of $K$, and therefore provides a measure of the oscillation of the reduced boundary of $E$ with respect to the boundary of $K$. Note that in the case $f$ constantly equal to one,
$\BPP$ agrees with $\beta_1$.

It is not immediately clear that \eqref{bpp} is the appropriate analogue of \eqref{betaB} in the anisotropic case. Noting that $x \mapsto (x-y)/f_*(x-y)$ is the radial projection of $\Rn \setminus \{0\}$ onto $\partial K + y$, one may initially want to consider the term
\begin{align}
\BPP^*(E)&
= \min_{y \in R^n} 
\bigg( \frac{1}{2n |K|^{1/n}|E|^{1/n'} }\int_{\partial^* E}
\Big|\nu_E(x) - \nu_{K}\Big(\frac{x-y}{f_*(x-y)}\Big) \Big|^2\, d\Hn(x) \bigg)^{1/2}\nonumber \\
& = \min_{y \in R^n} 
\bigg( \frac{1}{n|K|^{1/n}|E|^{1/n'} }\int_{\partial^* E}
1-\nu_E(x) \cdot \nu_K\Big(\frac{x-y}{f_*(x-y)}\Big)\, d\Hn(x) \bigg)^{1/2}.\label{betastar}
\end{align}
However, in Section~\ref{other} we see that such a term does not admit any stability result for general $\PP$. Indeed, in Example~\ref{crystalexample}, we construct a sequence of crystalline surface tensions that show that there does not exist a power $\sigma$ such that 
\begin{equation}\label{noway1}
\beta^*_{\PP}(E)^{\sigma} \leq C(n,f) \delta_{\PP}(E)
\end{equation}
for all sets $E$ of finite perimeter with $0<|E|<\infty$ and for all $\PP$. Furthermore, Example~\ref{pexample} shows that even if we restrict our attention to surface energies which are $\g$-$\lambda$ convex, a weaker notion of $\lambda$-ellipticity introduced in Definition~\ref{gammalambda}, an inequality of the form \eqref{noway1} cannot hold with an exponent less than $\sigma = 4.$
The examples in Section~\ref{other} illustrate the fact that, in the anisotropic case, measuring the alignment of normal vectors in a Euclidean sense is not suitable for obtaining a stability inequality for general $\PP$; it is essential to account for the anisotropy in this measurement. The $\PP$-oscillation index $\beta_{\PP}(E)$ in \eqref{bpp} does exactly this. 
 
In the positive direction, when the surface tension $f$ is $\g$-$\lambda$ convex, $\beta^*_{\PP}(E)$ is controlled by $\BPP(E)$. As one expects from Example~\ref{pexample}, the exponent in this bound depends on the $\g$-$\lambda$ convexity of $f$. We now define $\g$-$\lambda$ convexity.
\begin{definition}\label{gammalambda}
Let $f:\Rn \to \mathbb{R}$ be a nonnegative, convex, 
positively one-homogeneous function. 
 Then we say that $f$ is \textit{$\g$-$\lambda$ convex} for $\g \geq 0 , \lambda>0$ if
\begin{equation}\label{gamma lambda}f(\nu + \tau ) + f(\nu- \tau) -2f(\nu)
\geq \frac{\lambda}{|\nu|}
\left| \tau -\Big(\tau \cdot \frac{\nu}{|\nu|}\Big)\frac{\nu}{|\nu|}\right|^{2+\g}
\end{equation}
for all $\nu, \tau \in \Rn$ such that $\nu \neq 0$.
\end{definition}
Dividing \eqref{gamma lambda} by $\tau^2$, the left hand side  gives a second difference quotient of $f$. While $\lambda$-ellipticity assumes that $f\in C^2(\Rn \setminus \{0\})$ and that its second derivatives in directions $\tau$ that are orthogonal to $\nu$ are bounded from below, $\gamma$-$\lambda$ convexity only assumes that the second difference quotients in these directions have a bound from below that degenerates as $\tau$ goes to $0$. Of course, a $0$-$\lambda$ convex surface tension $f$ with $f\in C^2(\Rn\setminus \{0\})$ is $\lambda$-elliptic.
The $\ell^p$ norms $f_p (x) = (\sum_{i=1}^n |x_i|^p)^{1/p}$ for $p \in (1, \infty)$ are examples of $\gamma$-$\lambda$ convex surface tensions; see Section~\ref{other}.
When $f$ is a $\g$-$\lambda$ convex surface tension, the following theorem shows that $\BPP$ controls $\beta^*_{\PP}$. 
\begin{theorem}\label{th2} 
Let $f$ be a $\g$-$\lambda$ convex surface tension.
Then there exists a constant $C$ depending on $ \g, \lambda,$ and $m_{\PP}/M_{\PP}$ such that 
$$\BPP^*(E)^{(2+\g)/2} \leq C \left(\frac{P(E)}{n|K|^{1/n}|E|^{1/n'}}\right)^{\g/4} \beta_{\PP}(E).$$
for any set of finite perimeter $E$ with $0<|E|<\infty$. 
\end{theorem}
As in Theorem~\ref{Smooth}, the constant depends on $\m $ and $\M$ which are defined in \eqref{mM}.
As an immediate consequence of Theorem~\ref{th2}, Theorem~\ref{th1}, and Theorem~\ref{Smooth}, we have the following result. 
\begin{corollary}\label{cor1} 
If $f$ is a $\g$-$\lambda$ convex surface tension, then there exists a constant $C$ depending on $n, \g, \lambda,$ and $ m_{\PP}/M_{\PP}$ such that 
$$
\alpha_{\PP}(E)^2 + \BPP^*(E)^{\sigma} \leq C  \left(\frac{P(E)}{n|K|^{1/n}|E|^{1/n'}}\right)^{\g n/(n+1)} \delta_{\PP}(E)$$
for any set of finite perimeter $E$ with $0<|E|<\infty$, where $\sigma=2n(2+\g)/(n+1)$. 

If $f$ is a $\lambda$-elliptic surface tension, then there exists a constant $C$ depending on $n, \g, \lambda, m_{\PP}/M_{\PP},$ and $\| \nabla^2 f\|_{C^0(\pa K)}$ such that 
$$
\alpha_{\PP}(E)^2 + \BPP^*(E)^{2} \leq C \delta_{\PP}(E)$$
for any set of finite perimeter $E$ with $0<|E|<\infty$.
\end{corollary}

\subsection{Discussion of the proofs}
At the core of the proof of \eqref{FuscoJulin} are a selection principle argument, the regularity theory of almost-minimizers of perimeter, and an analysis of the second variation of perimeter. Indeed, with a selection principle argument in the  spirit of the proof of (\ref{FMPStatement}) by Cicalese and Leonardi in \cite{CiLe12}, Fusco and Julin reduce to a sequence $\{F_j\}$ such that each $F_j$ is a $(\Lambda, r_0)$-minimizer of perimeter (Definition~\ref{pm}) and $F_j \to B$ in $L^1$.
 Then, by the standard regularity theory, each set $F_j$ has boundary  given by a small $C^1$ perturbation of the boundary of the ball.  This case is handled by a theorem of Fuglede in \cite{fuglede1989}, which says the following:
Let $E$ be a \textit{nearly spherical set}, i.e., a set with barycenter $\bary E = |E|^{-1}\int_E  x\,dx$ at the origin such that $|E| = |B|$ and 
$$\partial E = \{ x + u(x) x : x \in \partial B\}$$
for $u: \partial B \to \mathbb{R}$ with $u\in C^1(\partial B)$. There exist $C$ and $\e$ depending on $n$ such that if $\| u \|_{C^1(\partial B)} \leq\e$, then 
\begin{equation}\label{fugst}
\|u\|_{H^1(\partial B)}^2 \leq C \delta_1(E).
\end{equation}
The proof of \eqref{fugst} makes explicit use of spherical harmonics to provide a lower bound for the second variation of perimeter. 
It is then easily shown that $\alpha_1(E) + \beta_1(E) \leq C \|u\|_{H^1(\partial B)},$ and therefore (\ref{fugst}) implies \eqref{FuscoJulin}
in the case of nearly spherical sets.
Indeed, $\alpha_1(E) \leq C \beta_1(E)$ as shown in Proposition \ref{prop1}, and in the case of nearly spherical sets, the oscillation index $\beta_1$ is essentially an $L^2$ distance of gradients: if $w(x) = x+ u(x)x$, then 
$$\nu_E(w(x)) = \frac{ x(1+ u(x)) + \nabla u(x)}{\sqrt{(1+u)^2 + |\nabla u|^2}},$$
where the $\nabla u$ is the tangential gradient of $u$. Then
\begin{align*} 
n|K|\beta_1(E)^2 \leq \int_{\partial E} 1 - \nu_E(w) \cdot \frac{w}{|w|}\, d \Hn 
&= \int_{\partial B} \sqrt{ (1+u)^2 + |\nabla u|^2} - (1+u) \,d \Hn\\
&= \int_{\partial B} \frac{1}{2}| \nabla u|^2 + O(|\nabla u |^2)\, d\Hn \leq \|u\|_{H^1(\partial B)}^2.
\end{align*}
\vspace{2mm}

In each of Theorems~\ref{th1}, \ref{Smooth}, and \ref{dim2}, at least one of the three key ingredients of the proof of Fusco and Julin is missing. The proof of Theorem \ref{th1} uses a selection principle to reduce to a sequence of $(\Lambda, r_0)$-minimizers of $\PP$ converging in $L^1$ to $K$. However, for an arbitrary surface tension, uniform density estimates (Lemma~\ref{UDE}) are the strongest regularity property that one can hope to extract. We pair these estimates with \eqref{FiMPStatement} to obtain the result.

The proof of Theorem~\ref{Smooth} follows a strategy similar
 to that of the proof of (\ref{FuscoJulin}) in \cite{fuscojulin11}. 
If $f$ is a $\lambda$-elliptic surface tension, then $(\Lambda, r_0)$-minimizers of the corresponding surface energy $\PP$ enjoy strong regularity properties. Using a selection principle argument and the regularity theory, 
we reduce to the case where $\pa E$ is a small $C^1$ perturbation of $\pa K$. The difficulty arises, however, in showing the following analogue of Fuglede's result (\ref{fugst}) in the setting of the anisotropic surface energy.
\begin{proposition} \label{Fug-type}
Let $f$ be a $\lambda$-elliptic surface tension with corresponding surface energy $\PP$ and Wulff shape $K$. 
Let $E$ be a set such that $|E| = |K|$ and $\bary E = \bary K$, where $\bary E  =|E|^{-1} \int_E x\, dx$ denotes the barycenter of $E$. Suppose 
$$\partial E = \{ x + u(x) \nu_K(x) : x \in \partial K\}$$
where $u: \partial K \to \mathbb{R}$ is in $ C^1(\partial K)$. There exist $C$ and $\e_1$ depending on $n, \lambda,$ and $\m/\M$  such that if $\| u \|_{C^1(\partial K)}\leq \e_1$, then
 \begin{equation}\label{fugnew}\|u\|_{H^1(\partial K)}^2 \leq C\delta_{\PP}(E).
\end{equation}
\end{proposition}
Again, $\m$ and $\M$ are defined in \eqref{mM}.
To prove \eqref{fugst}, Fuglede shows that, due to the volume and barycenter constraints respectively, the function $u$ is orthogonal to the first and second eigenspaces of the Laplace operator on the sphere. This implies that, thanks to a gap in the spectrum of this operator, functions satisfying these constraints satisfy a Poincar\'{e} inequality with a larger constant than the Poincar\'{e} inequality that holds for $w \in H^{1}(\pa B)$ with mean zero (i.e., satisfying only the volume constraint). Fuglede's reasoning uses that fact that the eigenvalues and eigenfunctions of the Laplacian on the sphere are \textit{explicitly known}.

The analogous operator on $\pa K$ arising in the second variation of $\PP$ also has a discrete spectrum, but one cannot expect to understand its spectrum explicitly. Instead, to prove \eqref{Fug-type}, we  exploit \eqref{FiMPStatement} in order to obtain a Poincar\'{e} inequality with a larger constant for functions $u\in H^{1}(\pa K)$ satisfying the volume and barycenter constraints. 

Then, as in the isotropic case, one shows that $ \alpha_{\PP}(E) + \BPP(E) \leq C \|u \|_{H^1(\partial K)}$ for a constant $C = C(n, \| \nabla^2 f\|_{C^0(\pa K)})$, and therefore \eqref{fugnew} implies \eqref{smooth12} for small $C^1$ perturbations. Indeed, Proposition~\ref{prop1} implies that  $\alpha_{\PP}(E) \leq C(n) \BPP(E)$, and the fact that  $ \BPP(E) \leq C \|u\|_{H^1(\partial K)}$
 is a consequence of a Taylor expansion and a change of coordinates. The computation is postponed until (\ref{betaexp}) as it relies on notation introduced in Section~\ref{smoothsection}.
 
The proof of Theorem~\ref{dim2} also uses a selection principle-type argument to reduce to a sequence of almost-minimizers of $\PP$ converging in $L^1$ to the Wulff shape. In this case, a rigidity result of Figalli and Maggi in \cite{figallimaggi11} allows us reduce to the case where $E$ is a convex polygon whose set of normal vectors is equal to the set of normal vectors of $K$. From here, an explicit computation (Proposition~\ref{crystalprop}) shows the result. 
\\
 
 The paper is organized as follows. In Section \ref{preliminaries}, we introduce some necessary preliminaries for our main objects of study. Section \ref{generalsection} is dedicated to the proof of Theorem \ref{th1}, while in Sections \ref{smoothsection} and \ref{crystalsection} we prove Theorems \ref{Smooth}  and \ref{dim2} respectively. In Section~\ref{other}, we consider the term $\beta_{\PP}^*(E)$ defined in \eqref{betastar}, providing two examples that show that one cannot expect stability with a power independent of the regularity of $f$ and proving Theorem~\ref{th2}. \\

 \noindent {\bf{Acknowledgments:}} The author would like to thank Alessio Figalli and Francesco Maggi for their mentorship, guidance, and many helpful discussions. Further thanks are due to a thorough referee for providing several useful remarks. This research was supported by the NSF Graduate Research Fellowship under Grant No. DGE-$1110007$.
\section{Preliminaries}\label{preliminaries}
Let us introduce a few key properties about sets of finite perimeter, the anisotropic surface energy, and the $\PP$-oscillation index $\BPP$.
\subsection{Sets of finite perimeter}\label{SOFP}
Given an $\Rn$-valued Borel measure $\mu$ on $\Rn$, the \textit{total variation} $|\mu|$ of $\mu$ on a Borel set $E$ is defined by
$$|\mu|(E) = \sup \bigg\{ \sum_{j\in\mathbb{N}} |\mu(E_j) | : \ E_j\cap E_i = \emptyset, \ \bigcup_{j\in\mathbb{N}} E_j \subset E\bigg\}.$$
A measurable set $E\subset \Rn$ is said to be a \textit{set of finite perimeter} if the distributional gradient ${\rm{D}}\chi_E$ of the characteristic function of $E$ is an $\Rn$-valued Borel measure on $\Rn$ with $|{\rm{D}}\chi_E|(\Rn) < \infty.$ 

For a set of finite perimeter $E$, the \textit{reduced boundary} $\partial^*E$ is the set of points $x\in\Rn$ such that $|{\rm{D}}\chi_E|(B_r(x))>0$ for all $r>0$ and 
\begin{equation}\label{reducedbd}
\underset{r\to 0^+}{\lim} \frac{{\rm{D}}\chi_E (B_r(x))}{|{\rm{D}}\chi_E| (B_r(x) )} \quad \text{ exists and belongs to }S^{n-1}.
\end{equation}
If $x \in \partial^*E$, then we let $-\nu_E$ denote the limit in \eqref{reducedbd}. We then call $\nu_E: \partial^*E \to S^{n-1} $ the \textit{measure theoretic outer unit normal to E}. Up to modifying $E$ on a set of Lebesgue measure zero, one may assume that the topological boundary $\pa E$ is the closure of the reduced boundary $\pa^*E.$ For the remainder of the paper, we make this assumption.

\subsection{The surface tension and the gauge function}\label{Subsection:anisotropies}
Throughout the paper, we let 
\begin{equation}\label{mM}
m_{\PP} = \underset{\nu \in S^{n-1} }{\inf}f(\nu), \qquad M_{\PP}=  \underset{\nu \in S^{n-1} }{\sup}f(\nu).
\end{equation}
It follows that 
$$\frac{1}{M_{\PP}} = \underset{x \in S^{n-1}}{\inf}f_*(x), \qquad \frac{1}{m_{\PP} } =\underset{x\in S^{n-1}}{\sup}f_*(x).$$
One easily shows that $f(\nu) = \sup \{ x\cdot \nu : x \in K \} $ and $f_*(x) = \inf \{ \lambda: \frac{x}{\lambda} \in K\}.$ This also implies that $B_{m_{\PP}} \subset K \subset B_{M_{\PP}}$, and so if 
$|K| =1,$ then $\m^n |B| \leq 1 \leq \M^n |B|$.
As mentioned in the introduction, the surface tension $f$ and gauge function $f_*$ are dual in the sense that they satisfy a Cauchy-Schwarz-type inequality, called the \textit{Fenchel inequality}:
$$x\cdot \nu \leq f_*(x) f(\nu)$$
for all $x,\nu \in \Rn$.
We may characterize the equality cases in the Fenchel inequality: for any $\nu$, ${x}\cdot \nu={f_*(x)}f(\nu) $ if and only if $\nu$ is normal to a supporting hyperplane of $K$ at the point $\frac{x}{f_*(x)}\in \partial K$. Indeed, $\nu$ is normal to a supporting hyperplane of $K$ at $x \in \partial K$ if and only if $\nu \cdot (y- x) \leq 0$ (so $ \nu \cdot y \leq \nu \cdot x$)
for all $y \in K.$ This holds if and only if  $\nu \cdot x = \sup\{ y\cdot \nu : y\in K\} = f(\nu)$. 
In particular, if $x\in \partial^* K$, then $f_*(x) = 1$ and 
 \begin{equation}\label{bdK}
 f(\nu_K(x)) = x\cdot \nu_K(x).
 \end{equation}

We may compute the gradient of $f_*$ at points of differentiability using the Fenchel inequality.
  The gauge function $f_*$ is differentiable at $x_0\in \Rn$
   if there is a unique supporting hyperplane to $K$ at $\frac{x_0}{f_*(x_0)}\in \partial K$. 
 For such an $x_0, $
let $\nu_0 = \nu_K(\frac{x}{f_*(x)}) \in \Rn$ be normal to the supporting hyperplane to $K$ at $\frac{x_0}{f_*(x_0)}$, so $\frac{x_0}{f_*(x_0)}\cdot\nu_0 = f(\nu_0)$ by \eqref{bdK}.
We define the Fenchel deficit functional by $G(x) = f(\nu_0)f_*(x) - x\cdot \nu_0.$
By the Fenchel inequality, $G(x) \geq 0$ for all $x$ and $G(x_0) =0$, so  
$G$ has a local minimum at $x_0$ and thus
$$0 =\nabla G(x_0)
 =  f(\nu_0)\nabla f_*(x_0) -\nu_0.$$
Rearranging, we obtain $\nabla f_*(x_0) = \frac{\nu_0}{ f(\nu_0)}.$
The $1$-homogeneity of $f$ then implies that
\begin{equation}\label{equal1}
f(\nabla f_*(x)) =1.
\end{equation}
Furthermore, this implies that
\begin{equation}\label{gradprop}
x \cdot \nabla f_*(x) = x \cdot \nu_K\Big(\frac{x}{f_*(x)}\Big) = f_*(x)
\end{equation}
(alternatively, this follows from Euler's identity for homogeneous functions).
An analogous argument ensures that 
\begin{equation}\label{gradf}\nabla f(\nu_K(x)) = x \end{equation}
for $x\in\partial^* K$. Furthermore, we compute
\begin{equation}\label{div}
\DIV\frac{x}{f_*(x)} = \frac{n-1}{f_*(x)}.
 \end{equation}
Indeed,
\begin{align*}
\DIV \frac{x}{f_*(x)} 
& = \frac{ \tr \, \nabla x}{f_*(x)} + x\cdot \nabla\Big( \frac{1}{f_*(x)}\Big) =\frac{n}{f_*(x)} -\frac{x \cdot \nabla f_*(x) }{f_*(x)^2} = \frac{n-1}{f_*(x)},
\end{align*}
where the final equality follows from \eqref{gradprop}. 

\subsection{Properties of $\alpha_{\PP}$, $\BPP$, and $\g_{\PP}$}\label{abg} 
 Using the divergence theorem, by approximation and the dominated convergence theorem, and \eqref{div},
  we find that for any $y\in\Rn$,
$$
 \int_{\partial^* E } \frac{x-y}{f_*(x-y) }\cdot \nu_{E}(x) \,d\Hn = (n-1)\int_E \frac{dx}{f_*(x-y)} .
$$ 
We may then write 
 \begin{equation}\label{betaform}\BPP(E)^2   = \frac{\PP(E) -(n-1) \g_{\PP} (E)}{n|K|^{1/n}|E|^{1/n'}},
\end{equation}
 where $\g_{\PP}(E)$ is defined by
 \begin{equation}\label{gamma}
 \g_{\PP} (E) =  \underset{y\in\Rn}{\sup} \int_{E }
 \frac{dx}{f_*(x-y) }.
 \end{equation}
  The supremum in (\ref{gamma}) is attained, though
perhaps not uniquely. If $y\in \Rn$ is a point such that 
 $$\g_{\PP}(E) = \int_E \frac{dx}{f_*(x-y)},$$
 then we call $y$ a \textit{center of $E$}, and we denote by $y_E$ a generic center of $E$. The Wulff shape $K$ has unique center $y_K = 0$.
  Indeed, take any $y \in \Rn,$ $y\neq 0$, and recall that $K=
  \{f_*(x) < 1\} $. Then
 \begin{align*}
 \int_{K} \frac{dx }{f_*(x)}&  - \int_K \frac{dx}{f_*(x-y)} 
  = \int_K \frac{dx}{f_*(x)} - \int_{K+y} \frac{dx}{f_*(x)}\\
&  = \int_{K \setminus (K+y)} \frac{dx}{f_*(x)} - \int_{(K+y)\setminus K} \frac{dx}{f_*(x) } > \int_{K \setminus (K+y)}1dx - \int_{(K+y ) \setminus K } 1dx  =0.
  \end{align*}
A similar argument verifies that if $|E| = |K|$, then 
 \begin{equation}\label{maxgamma}
 \g_{\PP}(E) \leq \g_{\PP}(K).
 \end{equation}
 Moreover,  $(n-1) \g_{\PP}(K) = \PP(K) = n|K|.$

The following continuity properties of $\PP$ and $\g_{\PP}$ will be useful.
 \begin{proposition}\label{convergence} Suppose that $\{E_j\}$ is a sequence of sets converging in $L^1$ to a set $E$, and suppose that $\{f^j\}$ is a sequence of surface tensions converging locally uniformly to $f$, with corresponding surface energies $\{\PP_j\}$ and $\PP$.
\begin{enumerate}
\item The following lower semicontinuity property holds:
$$
\PP(E) \leq \underset{j\to\infty}{\liminf}\, \PP_j(E_j).
$$
\item The function $\g_{\PP}$ defined in (\ref{gamma}) is H\"{o}lder continuous with respect to $L^1$ convergence of sets with H\"{o}lder exponent equal to $1/n'$. In particular,
$$ |\g_{\PP}(E) - \g_{\PP}(F) | \leq \frac{n|K|}{n-1} |E\Delta F|^{1/n'}$$
for any two sets of finite perimeter $E, F\subset \Rn$. 
Moreover,
\begin{equation*}\label{gammacts}
\underset{j \to \infty}\lim \g_{\PP_j}(E_j) = \g_{\PP}(E).
\end{equation*}
\end{enumerate}
 \end{proposition}
  \begin{proof}
\noindent \textit{Proof of $(1)$:} From the divergence theorem and the characterization $f(\nu) = \sup \{ x\cdot \nu : f_*(x) \leq 1\}$, one finds that the surface energy of a set $E$ is the anisotropic total variation of its characteristic function $\chi_E$: 
\begin{equation}\label{TV}
\PP_j(E_j) = TV_{f_{j}}(\chi_{E_j}) : = \sup \Big\{ \int_{E_j} \DIV T\, dx\ \big{|} \ \ T\in C^1_c(\Rn, \Rn),\
 f_*^j(T) \leq 1 \Big\}.
\end{equation}
Let $T \in 
C^1_c(\Rn, \Rn)$ be a vector field such that
 $f_*(T) \leq 1$ for all $x\in \Rn$. Then we have
\begin{align*}
\int_E \DIV T \,dx &  = \lim_{j\to \infty} \int_{E_j} \DIV T\, dx = \lim_{j \to \infty} \|f_*^j (T)\|_{L^{\infty}(\Rn)} \int_{E_j} \DIV S_j\, dx \leq  \underset{j\to\infty}{\liminf} \mbox{ }\PP_j(E_j),
\end{align*}
where we take
 $S_j = 
T/
  \|f_*^j(T)\|_{L^{\infty}(\Rn)}.$ Taking the supremum over $\{T\in C^1_c(\Rn,\Rn ) : f_*(T) \leq 1\}$, we obtain the result.\\

\noindent \textit{Proof of $(2)$:} By \eqref{gamma},
\begin{align*} \g_{\PP}(E) - \g_{\PP}(F)& \leq \int_E \frac{dx}{f_*(x-y_E)} - \int_F \frac{dx}{f_*(x-y_E)} \leq \int_{E\Delta F } \frac{dx}{f_*(x-y_E)} .
\end{align*}
Letting $r$ be such that $|rK| = |E\Delta F|$ and  recalling \eqref{maxgamma}, we have
\begin{equation}\label{tozero}
\int_{E\Delta F } \frac{dx}{f_*(x-y_E)}  \leq \int_{rK} \frac{dx}{f_*(x)} =\g_{\PP}(rK) = \frac{\PP(rK)}{n-1} = \frac{n|K|r^{n-1}}{n-1} = \frac{n|K|}{n-1} |E\Delta F|^{1/n'}.
\end{equation}
Thus $\g_{\PP}(E) - \g_{\PP}(F) \leq \frac{n|K|}{n-1} |E\Delta F|^{1/n'}$. The analogous argument holds for $\g_{\PP}(F) - \g_{\PP}(E),$ implying the H\"{o}lder continuity of $\g_{\PP}$.

For the second equation, we note that if $f^j \to f$ locally uniformly, then $f_*^j \to f_* $ locally uniformly and $M_{\PP_j} \to M_{\PP}$. The triangle inequality gives
  \begin{align*}
  |\g_{\PP_j}(E_j) - \g_{\PP}(E) | \leq |\g_{\PP_j}(E_j) - \g_{\PP}(E_j) | + |\g_{\PP}(E_j) - \g_{\PP}(E)|.
  \end{align*} 
The second term goes to zero by the H\"{o}lder continuity that we have just shown. To bound the first term, let $y_{E_j}$ be a center of $E_j$ with respect to the surface energy $\PP_j$. If $\g_{\PP_j}(E_j) \geq \g_{\PP}(E_j) $, then 
  \begin{align*}
0\leq  \g_{\PP_j}(E_j) - \g_{\PP}(E_j) &\leq \int_{E_j} \frac{1}{f_*^j (x-y_{E_j})} - \frac{1}{f_*(x-y_{E_j})} dx 
  = \int_{E_j + y_{E_j}} \frac{1}{f_*^j(x)} - \frac{1}{f_*(x) } dx\\
   &= \int_{\Rn}\chi_{(E_j + y_{E_j})\setminus B_{\e}(0)} \bigg(\frac{1}{f_*^j(x)} - \frac{1}{f_*(x) } \bigg)dx + \int_{B_{\e}(0)} \frac{1}{f_*^j(x)} - \frac{1}{f_*(x) } dx .
 \end{align*}
 For $\e>0$ fixed, the first integral goes to zero as $j\to \infty$. For the second integral, we have
$$
  \int_{B_{\e}(0)} \frac{1}{f_*^j(x)  } + \frac{1}{f_*(x)} \,dx \leq \int_{B_{\e}(0)} \frac{M_{\PP_j} + M_{\PP}}{|x|}\, dx \leq C \e^{n-1}.
$$
Taking $\e \to 0$, we conclude that $ \g_{\PP_j}(E_j) - \g_{\PP}(E_j) \to0$ as $j \to \infty$. The case where $\g_{\PP_j}(E_j) \leq \g_{\PP}(E_j) $ is analogous.
  \end{proof}
  
\begin{remark}\label{RMK: LSC}\rm{With sequences as in the hypothesis of Proposition~\ref{convergence} above,  $\BPP$ has the following lower semicontinuity property:
$$
\BPP(E)\leq \underset{j\to\infty}{\liminf}\,\beta_{\PP_j}(E_j).
$$
This follows immediately from parts $(1)$ and $(2)$ of Proposition~\ref{convergence} and the decomposition in \eqref{betaform}.}
\end{remark}

 \begin{lemma}\label{3.1}
For every $\e>0$, there exists $\eta>0$ such that if $|F\Delta K | \leq \eta$, then $|y_F|<\e $ for any center $y_F$ of $F$.
 \end{lemma} 
\begin{proof}
Suppose $|K \Delta F_j| \to 0$. By the triangle inequality,
\begin{align*} 
\int_K \frac{dx}{f_*(x)}
& \leq  \bigg|\int_{K} \frac{dx}{f_*(x)} - \int_{F_j} \frac{dx}{f_*(x-y_{F_j})} \bigg|+\bigg| \int_{F_j} \frac{dx}{f_*(x -y_{F_j})} - \int_{K} \frac{dx}{f_*(x -y_{F_j})} \bigg|
  +\int_{K} \frac{dx}{f_*(x -y_{F_j})}.
  \end{align*}
By \eqref{tozero}, the first two terms on the right hand side go to zero as $j\to \infty,$ implying that
 $$ \int_K \frac{dx}{f_*(x)}\leq  \underset{j \to \infty}{\lim} \int_{K} \frac{dx}{f_*(x -y_{F_j})}.$$
Because $K$ has unique center $y_K = 0$, we conclude that $|y_{F_j} |\to 0.$
\end{proof}

We now introduce the relative surface energy and the anisotropic coarea formula.
 Given an open set $A$ and a set of finite perimeter $E$, the anisotropic surface energy of $E$ \textit{relative to $A$} is defined by  
$$\PP(E;A) = \underset{\partial^* E \cap A}{\int} f(\nu_E(x))\, d\Hn(x).$$
 For a Lipschitz function $u:\Rn \to \mathbb{R}$ and an open set $E$, the anisotropic coarea formula states that
$$\int_E f(-\nabla u(x))\,dx = \int_0^{\infty}\PP(\{ u>r\};E)\, dr.$$
The anisotropic coarea formula is proved in the same way as the coarea formula (see, for instance, \cite[Theorem 13.1]{maggi2012sets}), replacing the Euclidean norm with $f$ and $f_*$ and using \eqref{TV}.
When $u$ is bounded by a constant $C$ on $E$, then applying the anisotropic coarea formula to $w = C-u$ yields
\begin{align*}\int_E f(\nabla u(x))\,dx&= \int_E f(-\nabla w(x))\, dx =\int_0^C \PP(\{C-u>r\}; E) \,dt\\
&= \int_0^C \PP(\{u <C-r\} ; E) \,dr = \int_0^C \PP(\{u<r\} ; E) \, dr
\end{align*}
Moreover, approximating by simple functions, we may produce a weighted version: 
$$\int_E f(\nabla u(x)) g(f_*(x)) \,dx= \int_0^{\infty} \PP(\{u<r\};E) g(r) \,dr$$
whenever $g: \mathbb{R} \to[0,\infty] $ is a Borel function. We will frequently use this weighted version with $u(x) = f_*(x)$, $E$ a bounded set, and $g(r) = \frac{1}{r},$ which, using (\ref{equal1}), gives 
\begin{equation}\label{awcoarea}
\int_E \frac{dx}{f_*(x)} = \int_0^{\infty} \frac{\PP( \{f_*(x) <r\} ; E )}{r}\, dr= \int_0^{\infty} \frac{\PP( rK; E )}{r}\, dr.
\end{equation}

 We conclude this section with the following Poincar\'e-type inequality, which shows that $\BPP(E)$ controls $\alpha_{\PP}(E)$ for all sets $E$ of finite perimeter.
\begin{proposition}\label{prop1} \label{poincare}
 There exists a constant $C(n)$ such that if $E$ is a set of finite perimeter with $0<|E|<\infty$, then 
\begin{equation}\label{bd}
\alpha_{\PP}(E) + \delta_{\PP}(E)^{1/2} \leq C(n) \BPP(E).
\end{equation}
\end{proposition}
\begin{proof} 
We follow the proof of the analogous result for the perimeter in \cite{fuscojulin11}.
Due to the scaling and translation invariance of $\alpha_{\PP}, \BPP$, and $\delta_{\PP}$, 
we may assume that $|E| = |K|=1$ and that $E$ has center zero. We have 
$$ \g_{\PP}(K) - \g_{\PP}(E)=  \int_K \frac{dx}{f_*(x)}-\int_E \frac{dx}{f_*(x) } 
=\int_{K\setminus E} \frac{dx}{f_*(x)}-\int_{E\setminus K} \frac{dx}{f_*(x)} .
$$
Therefore, adding and subtracting $\PP(K)/n=(n-1)\g_{\PP}(K)/n$ in \eqref{betaform}, we have 
\begin{align*}
\BPP(E)^2 &
=\delta_{\PP}(E)
+\frac{n-1}{n}\left( \int_{K\setminus E} \frac{dx}{f_*(x)}  -
\int_{E\setminus K} \frac{dx}{f_*(x)} \right).
\end{align*}
We want to bound the  final two integrals from below by $\alpha_{\PP}(E)^2$.
To this end, we let $a:= |E\setminus K| = |K\setminus E|$ and 
define the $K$-annuli $A_{R,1} = K_R\setminus K$ 
and $A_{1,r}=K\setminus K_r$, where  $R> 1 >r$ 
are chosen such that $|A_{R,1} | =|A_{1,r} | =a$.
 In particular, $R = (1+a)^{1/n} $ and $r= (1-a)^{1/n}.$ 
 By (\ref{maxgamma}) and (\ref{awcoarea}),
\begin{align*}
\int_{K\setminus E} \frac{dx}{f_*(x) } 
\geq \int_{A_{1,r}} \frac{dx}{f_*(x) }
= \int_r^1 \frac{\PP(sK)}{s}\,ds
 = \int_r^1 ns^{n-2}\, ds = \frac{n}{n-1} [1-r^{n-1}]
\end{align*}
and 
\begin{align*}
\int_{E\setminus K} \frac{dx}{f_*(x) } 
\leq \int_{A_{R,1}} \frac{dx}{f_*(x)}
 = \int_1^R\frac{\PP(sK)}{s} \,ds
  = \int_1^R ns^{n-2} \,ds
 = \frac{n}{n-1}[R^{n-1} - 1].
\end{align*}
Subtracting the second from the first, we have 
$$ \frac{n-1}{n} \left(\int_{K\setminus E} \frac{dx}{f_*(x)}  -\int_{E\setminus K} \frac{dx}{f_*(x)}\right)\geq 2 - r^{n-1} - R^{n-1}.$$

The function $g(t) = (1+t)^{1/n'}$ is function is strictly concave, with $\frac{1}{2}(g(t) + g(s) ) \leq g(\frac{t}{2} + \frac{s}{2}) - C|t-s|^2$, and therefore
$2 - [ (1+a )^{1/n'} + (1-a)^{1/n'}] \geq 8C|a|^2.$ Thus 
\begin{align*}
\BPP(E)^2  \geq \delta_{\PP}(E) +  [ 2 - (1-a)^{1/n'} - (1+a)^{1/n'}]  &\geq \delta_{\PP}(E) + 8C|a|^2 =  \delta_{\PP} +2C \left(|E\setminus K| + |K \setminus E|\right)^2\\
& = \delta_{\PP}(E) +2C |K\Delta E|^2\geq \delta_{\PP}(E)  +2C\alpha_{\PP}(E)^2.
\end{align*}
\end{proof}

\section{Stability for General Anisotropic Surface Energy $\PP$}\label{generalsection}

In this section, we prove Theorem~\ref{th1}. We begin by introducing a few lemmas that are needed the proof.
The first allows us to reduce the problem to sets contained in some fixed ball.
\begin{lemma}\label{compact}
 There exist constants $R_0>0$ and $C>0$ depending only on $n$ and $M_{\PP}$
such that, given a set of finite perimeter $E$ with $|E|= |K|$, 
 we may find a set $E'$ such that $|E' | = |K|, \ E'\subset B_{R_0}$,
  and 
\begin{equation}\label{BOUND}\BPP(E)^2 
\leq \BPP(E')^2 + C \delta_{\PP}(E), \mbox{ 		} \qquad
\delta_{\PP}(E') \leq C \delta_{\PP}(E).
\end{equation}
\end{lemma}
\begin{proof}
A simple adaptation of the proof of \cite[Theorem~$4.1$]{maggi2008some} 
ensures that we may find constants 
$\delta_0, C_0,$ $C_1$, and $\tilde{R}_0$ depending on $n$ and $M_{\PP}$ 
such that $C_0\delta_0<1/2$ and the following holds: if
 $\delta_{\PP}(E) \leq \delta_0,$ 
then there exists a set $\tilde{E}\subset E$ such that $\tilde{E} \subset B_{\tilde{R}_0}$ and 
\begin{equation}\label{cut}
|\tilde{E}| \geq |K|( 1 -  C_1 \delta_{\PP}(E)), \qquad \PP(\tilde{E}) \leq \PP(E) + C_0\delta_{\PP}(E) |E|^{1/n'}.
\end{equation}

If $\delta_{\PP}(E) > \delta_0$, then
$$\BPP^2 (E)
\leq \frac{\PP(E)}{n|K|}
 = \delta_{\PP}(E)+1 
\leq \frac{1+\delta_0}{\delta_0} \delta_{\PP}(E).$$
Simply taking $E' = K$,
 we have $\delta_{\PP}(E')
 \leq \delta_{\PP}(E)$
 and 
$\BPP(E)^2 \leq \frac{1+\delta_0}{\delta_0} \delta_{\PP}(E)$, proving \eqref{BOUND}.

On the other hand, if $\delta_{\PP}(E) \leq \delta_0$, let $E' = r \tilde{E}$ with $r\geq 1$ such that $|E'|=|r\tilde{E}|=|E|.$ By \eqref{betaform},
\begin{equation}\label{Bounded set}
\begin{split}
\BPP(E)^2 -\BPP(E')^2 
& = \frac{\PP(E) - \PP(E')}{n|K|} 
+ \frac{n-1}{n|K|} \left( \g_{\PP}(E' ) - \g_{\PP}(E) \right) \\
&\leq \delta_{\PP}(E)
+ \frac{n-1}{n|K|} \left(r^{n-1} \g_{\PP}(\tilde{E}) - \g_{\PP}(E) \right).
\end{split}
\end{equation}
Since $\tilde{E}\subset E$, $\g_{\PP}(\tilde{E} ) \leq \g_{\PP}(E)$, which implies that
\begin{align*} 
 \frac{n-1}{n|K|} \left(r^{n-1} \g_{\PP}(\tilde{E}) - \g_{\PP}(E) \right)\leq 
  \frac{n-1}{n|K|} (r^{n-1} -1) \g_{\PP}(E).
\end{align*}
By \eqref{maxgamma} and the fact that $\g_{\PP} (K)= n|K|/(n-1)$,
\begin{align*}
    \frac{n-1}{n|K|} (r^{n-1} -1) \g_{\PP}(E)& \leq   \frac{n-1}{n|K|}(r^{n-1} -1) \g_{\PP}(K)
=r^{n-1} -1,
\end{align*}
and since $r\geq 1,$ 
$$ r^{n-1} -1\leq r^n-1=\frac{ |E| - |\tilde{E}|}{|\tilde{E}|}.$$
The first part of $(\ref{cut})$ implies that 
$$\frac{|E| - |\tilde{E}| }{|\tilde{E}|} \leq \frac{C_1 \delta_{\PP}(E)}{1 -C_1 \delta_{\PP}(E)}\leq \frac{C_1}{1-C_1\delta_0} \delta_{\PP}(E).$$
We have therefore shown that 
$$  \frac{n-1}{n|K|} \left(r^{n-1} \g_{\PP}(\tilde{E}) - \g_{\PP}(E) \right) \leq  \frac{C_1}{1-C_1\delta_0} \delta_{\PP}(E);$$
this together with \eqref{Bounded set} concludes the proof of the first claim in \eqref{BOUND}.

In the direction of the second claim in \eqref{BOUND}, the first and second parts of \eqref{cut} respectively imply that 
$$\PP(E')  = r^{n-1} \PP(\tilde{E})
 \leq \frac{\PP(\tilde{E}) }{(1-C_1 \delta_{\PP}(E))^{1/n'}} \leq \frac{\PP(E)+ C_0 \delta_{\PP}(E) |E|^{1/n'}}{(1-C_1 \delta_{\PP}(E))^{1/n'}}.$$
A Taylor expansion in $\delta_{\PP}(E)$ of the right hand side shows that 
\begin{align*}
\PP(E')  & \leq  \PP(E)+ C_0 \delta_{\PP}(E) |E|^{1/n'} +\frac{n-1}{n}C_1 \delta_{\PP}(E)\PP(E)+ O(\delta_{\PP}(E)^2)
\\
&
\leq \PP(E) + C \delta_{\PP}(E) \PP(E)
\end{align*}
for $\delta_0$ chosen sufficiently small.
Thus 
$$\delta_{\PP}(E') = \frac{\PP(E') - \PP(K)}{n|K|} \leq 
 \frac{\PP(E') - \PP(E)}{n|K|}
 \leq \frac{C\PP(E) \delta_{\PP}(E)}{n|K|} \leq C \delta_{\PP}(E),$$
 since $\PP(E) \leq \PP(K) +n|K| \delta_0$.
Finally, since $\tilde{E} \subset B_{\tilde{R}_0}$ and $E' = r \tilde{E}$ with $r \leq 1/(1-C_1\delta_0)^{1/n}$, we have $E' \subset B_{R_0}$ for $R_0 = r \tilde{R}_0$.
\end{proof}

Let us now consider the functional 
\begin{equation}\label{Q}Q(E) = \PP(E) +  \frac{|K|m_{\PP}}{8M_{\PP}}\left| \BPP(E)^2 - \e^2\right| + \Lambda\big| |E| - |K| \big| ,
\end{equation}
with $0<\e<1$ and $\Lambda>0$.
\begin{lemma}\label{LSC} 
A minimizer exists for the problem 
\begin{equation*}\label{Minimization}\min \left\{ Q(E) \, : \, E \subset B_{R_0}\right\}
\end{equation*}
for $\Lambda > 4n$ and $\e>0 $ sufficiently small. Moreover, any minimizer $F$ satisfies 
\begin{equation}\label{Perimeter Bounds} 
|F| \geq \frac{|K|}{2}, \qquad \PP(F)  \leq 2n|K|.
\end{equation}

\end{lemma}
\begin{proof} Let $\overline{Q} = \inf\{Q(E) :  E \subset B_{R_0} \}$, 
and let $\{F_j\}$ be a sequence such that $Q(F_j) \to \overline{Q}$. Since $F_j \subset B_{R_0}$ and $\PP(F_j)<2\overline{Q}$ for $j$ large enough, up to a subsequence, $F_j \to F$ in  $L^1$ for some $F \subset B_{R_0}$. The lower semicontinuity of $\PP$ (Proposition~\ref{convergence}(1)) ensures that $\PP(F) <\infty$.

We first show that $|F| \geq\frac{|K|}{2}$. For any $\eta>0$, $Q(F_j) \leq \overline{Q} + |K|\eta$ for $j$ sufficiently large.
 Furthermore, $\overline{Q} \leq Q(K) = \PP(K) + \frac{\e^2 |K|\m}{8\M}$, so 
$$\big| |F_j| - |K| \big| \leq \frac{1}{\Lambda}\left( \PP(K) +|K|\eta +\frac{\e^2 |K|\m}{8\M} \right)
=  \frac{|K|}{\Lambda} \left(n + \eta +\frac{\e^2\m}{8\M} \right)\leq  \frac{|K|}{2}$$
for $\e$ and $\eta$ sufficiently small. Therefore $|F_j|\geq \frac{|K|}{2},$ implying that $|F| \geq \frac{|K|}{2}$ as well.

We now show that $\liminf Q(F_j) \geq Q(F),$ so $F$ is a minimizer.
Recalling (\ref{betaform}), we have 
 \begin{align*}
 Q(F_j)  
	&= \PP(F_j) + \frac{|K|\m}{8\M} \bigg| \frac{\PP(F_j)-(n-1)\g_{\PP}(F_j) }{n|K|^{1/n}|F_j|^{1/n'}} - \e^2 \bigg| + \Lambda \big| |F_j| - |K|\big|
  \\
 &  \geq 
  \PP(F_j ) +\frac{|K|\m}{8\M} \bigg| \frac{\PP(F)-(n-1)\g_{\PP}(F_j)}{n|K|^{1/n}|F_j|^{1/n'}} - \e^2 \bigg|
  -\frac{|K|\m}{8\M} \bigg|\frac{ \PP(F_j) - \PP(F)}{n|K|^{1/n}|F_j|^{1/n'}} \bigg|  + \Lambda \big| |F_j|- |K|\big|.
  \end{align*}
    Let $a={\liminf}\mbox{ }\PP(F_j)$. Up to a subsequence, we may take this limit infimum to be a limit.
   By the lower semicontinuity of $\PP$, $a \geq \PP(F)$. Furthermore, $\g_{\PP}$ is continuous by Proposition~\ref{convergence}$(2)$, so
\begin{align*}
\underset{j\to \infty}{\liminf}\mbox{ } Q(F_j)& \geq Q(F) + \left(a-\PP(F)\right) - \frac{|K|^{1/n'}\m}{8n|F|^{1/n'}\M}\left|a - \PP(F)\right| \\
& = Q(F) + \left(a-\PP(F)\right)\Big( 1  - \frac{|K|^{1/n'}\m}{8n|F|^{1/n'}\M }\Big)\\
&\geq Q(F) + \left(a-\PP(F)\right)\Big(1  - \frac{2^{1/n'}\m}{8n\M }\Big)
\geq Q(F).
\end{align*}
Finally, $\e<1$ and therefore $\PP(F) \leq Q(F) \leq Q(K) \leq 2n|K|.$
\end{proof}

The following lemma shows that a minimizer of \eqref{Q} satisfies uniform density estimates.
\begin{lemma}\label{UDE}
Suppose $F$ is a minimizer of  $Q(E)$ as defined in (\ref{Q}) among all sets $E\subset B_{R_0}$.
Then there exist $r_0>0$ depending on $n, \Lambda,$ and $|K|$ and  $0<c_0<1/2$ depending on $n$ and $\Lambda$ 
 such that for any $x \in \partial^* F$ and for any $r<r_0$, 
\begin{equation}
\label{UDES}
\frac{c_0\m^n}{\M^n} \, \omega_n r^n \leq |B_r(x) \cap F| \leq \bigg(1-\frac{c_0 \m^n}{ \M^n} \bigg)\, \omega_n r^n.
\end{equation}
\end{lemma}
\begin{proof} We follow the standard argument for proving uniform density estimates for minimizers of perimeter functionals; see, for example, \cite[Theorem $16.14$]{maggi2012sets}. The only difficulty arises when handling the term $\frac{|K|\m}{8\M} |\BPP(E)^2 - \e^2 |$ in $Q(E)$, as it scales like the surface energy. 

For any 
 $x_0 \in \partial^* F$, let $r<r_0$, where $r_0$ is to be chosen later in the proof and $r$ is chosen such that 
\begin{equation}\label{lineup}\Hn(\partial^* F \cap \partial B_r(x_0)) =0. \end{equation}
This holds for almost every $r>0$. Note that if \eqref{UDES} holds for almost every $r<r_0$, then it must hold for all $r<r_0$ by continuity; it is therefore enough to consider $r$ such that \eqref{lineup} holds. 
Let 
$G = F \setminus B_r(x_0)$. For simplicity, we will use the notation $B_r$ for $B_r(x_0).$
Because $F$ minimizes $Q$, 
$$\PP(F) + \frac{|K|\m}{8\M} 
\left|\BPP(F)^2 - \e^2 \right| 
+ \Lambda \big| |F| - |K| \big| 
\leq \PP(G)
+\frac{|K|\m}{8\M} \left|\BPP(G)^2 - \e^2\right| 
+ \Lambda \big| |G| - |K| \big| ,$$
and so rearranging and using the triangle inequality, we have
$$\PP(F) \leq \PP(G) 
+ \frac{|K|\m}{8\M} \left| \BPP(F)^2  - \BPP(G)^2 \right| 
+ \Lambda |F \cap B_r|.$$
We subtract $\PP(F; \Rn\setminus B_r)$ from both sides; this is the portion of the surface energy where $\partial^*F$ and $\partial^*G$ agree. We obtain 
\begin{align} \label{udeB}
\PP(F ; B_r) &
 \leq
  \int_{\partial B_r \cap F} f(\nu_{B_r})\, d\Hn
 +  \frac{|K|\m}{8\M} \left|\BPP(F)^2-\BPP(G)^2 \right| 
  + \Lambda |F \cap B_r|.
  \end{align}
Indeed, this holds because \eqref{lineup} implies that 
$$
\PP(G) = \PP(F; \Rn \setminus B_r) + \int_{\partial B_r \cap F} f(\nu_{B_r}) \,d\Hn.
$$

We must control the term $\frac{|K|\m}{8\M} \left|\BPP(F)^2-\BPP(G)^2 \right|$ and require a sharper bound than the one obtained using H\"{o}lder continuity of $\g_{\PP}$ shown in Proposition~\ref{convergence}(2). Indeed, we must show that the only contributions of this term are perimeter terms that match those in \eqref{udeB} and terms that scale like the volume and thus behave as higher order perturbations. We have 
\begin{align*}
\left| \BPP (F)^2  - \BPP(G)^2 \right| & = \frac{1}{n|K|^{1/n}} \left| \frac{\PP(F) - (n-1) \g_{\PP}(F)}{|F|^{1/n'} } - \frac{ \PP(G) - (n-1) \g_{\PP}(G)}{|G|^{1/n'}} \right| \\
& \leq \frac{2\PP(F)}{n|K|^{1/n}} \big||F|^{-1/n'} -|G|^{-1/n'} \big| 
+  \frac{\left| \PP(F) - \PP(G) \right| +(n-1) \left| \g_{\PP}(F) - \g_{\PP}(G) \right|}{n|K|^{1/n}|G|^{1/n'}}.
\end{align*}
The function $v(z)= 1 - (1-z)^{1/n'}$ is convex and increasing with $v(1) = 1$, hence $v(z) \leq z$ for $z\in[0,1]$. Thus, as $|G| = |F| - |F\cap B_r|,$  
\begin{equation}\label{aaa}\big| |F|^{-1/n'}  - |G|^{-1/n'} \big| =  |G|^{-1/n'}\left(1 - \left( 1- \frac{|F\cap B_r|}{|F|}\right)^{1/n'}\right) \leq \frac{|F\cap B_r|}{|G|^{1/n'}|F| }.
\end{equation}
Since $2|F| \geq |K|$ by \eqref{Perimeter Bounds}, $4|G| \geq |K| $ for $r_0$ sufficiently small depending on $n$, so the right hand side of  \eqref{aaa} is bounded by 
$8|K|^{-1-1/n'} |F\cap B_r| .$
The coefficient $\frac{2\PP(F)}{n|K|^{1/n}}$ is bounded by $4|K|^{1/n'}$ thanks to \eqref{Perimeter Bounds}, so
\begin{equation}\label{D1}\frac{2\PP(F)}{n|K|^{1/n}} \big||F|^{-1/n'} -|G|^{-1/n'} \big| 
\leq 32 |K|^{-1}|F \cap B_r|.
\end{equation}
Therefore, by \eqref{D1} and again using the facts that $4|G| \geq |K|$, $2|F|\geq |K|$, and $\m/\M \leq 1,$ we have shown that
\begin{equation}\label{bbound}
\frac{|K| \m}{8\M}| \BPP(F)^2 - \BPP(G)^2| \leq 4|F\cap B_r|+
 \frac{|\PP(F) -\PP(G)|}{2n} + \frac{\m}{\M}\frac{n-1}{2n} |\g_{\PP}(F) - \g_{\PP}(G)|.
\end{equation}
For the term $\left| \PP(F) - \PP(G) \right| $, using \eqref{lineup}, we have
\begin{equation}\label{D2}\left| \PP(F) - \PP(G) \right| = \left| \int_{\pa^* F} f(\nu_F) \,d\Hn - \int_{\pa^* G} f(\nu_G) \,d \Hn \right|\leq \PP(F; B_r) + \int_{\pa B_r \cap F} f(\nu_{B_r})\, d\Hn,
\end{equation}
using \eqref{lineup} and the fact that $\pa^*F$ and $\pa^*G$ agree outside of $B_r$.
Similarly, for the term  $\left| \g_{\PP}(F) - \g_{\PP}(G) \right|$,  when $\g_{\PP}(F) \geq \g_{\PP}(G)$, thanks to \eqref{lineup} we have
\begin{align*}
 \g_{\PP}(F) - \g_{\PP}(G) &
\leq \int_{\pa^* F } \frac{(x-y_F)\cdot \nu_F(x) }{f_*(x-y_F)}\, d\Hn - \int_{\pa^* G} \frac{(x-y_F)\cdot \nu_G(x)}{f_*(x-y_F)}\,  d\Hn\\
&\leq \frac{M_{\PP}}{m_{\PP}} \bigg( \PP(F; B_r) + \int_{\pa B_r \cap F} f(\nu_{B_r})\, d\Hn\bigg).
\end{align*}
The analogous inequality holds when $\g_{\PP}(G) \geq \g_{\PP}(F)$, so
\begin{equation}\label{D3}
| \g_{\PP}(F) - \g_{\PP}(G)|\leq \frac{M_{\PP}}{m_{\PP}} \bigg( \PP(F; B_r) + \int_{\pa B_r \cap F} f(\nu_{B_r})\, d\Hn\bigg).
\end{equation}
Combining \eqref{bbound}, \eqref{D2}, and \eqref{D3}, we have shown
\begin{equation}\label{udeC}
\frac{|K| \m}{8\M}
\left| \BPP (F)^2  - \BPP(G)^2 \right|  \leq 
4 |F \cap B_r|+ \frac{1}{2}  \bigg( \PP(F; B_r) + \int_{\pa B_r \cap F} f(\nu_{B_r})\, d\Hn\bigg).
\end{equation}
Combining (\ref{udeB}) and (\ref{udeC}) and rearranging, we have 
\begin{align*}\frac{1}{2} \PP( F; B_r)  &
\leq \frac{3}{2} \int_{\partial B_r \cap F} f(\nu_{B_r})\,d\Hn 
+ (4+ \Lambda )\left| F \cap B_r\right|.
\end{align*}
Proceeding in the standard way, we add the term $\frac{1}{2}\int_{\partial B_r \cap F} f(\nu_{B_r} )\,d\Hn $ to both sides, which gives
\begin{align*}
\frac{1}{2} \PP(F\cap B_r) &
\leq 2 \int_{\partial B_r\cap F} f(\nu_{B_r })\,d\Hn +(4+ \Lambda) \left|F \cap B_r\right|.
\end{align*}
By the Wulff inequality, $\PP(F\cap B_r) \geq n|K|^{1/n} |F\cap B_r|^{1/n'}$, and for $r_0$ small enough depending on $n, \Lambda$, and $|K|$, we may absorb the last term on the right hand side to obtain
\begin{align}\label{density}
\frac{n|K|^{1/n}|F\cap B_r|^{1/n'}}{4} 
\leq 2 \int_{\partial B_r\cap F} f(\nu_{B_r })\,d\Hn.
\end{align} 
Let $u(r) = |F \cap B_r|$, and thus $u'(r) = \Hn (\partial B_r \cap F)$, so the right hand side above is bounded by $2 M_{\PP} u'(r)$. Furthermore, $|K|^{1/n}\geq \m$,
so \eqref{density} yields the differential inequality 
 $$\frac{ n\m}{8\M}  \leq u'(r) u(r)^{-1/n'} = n (u^{1/n})'.$$
 Integrating these quantities over the interval $[0,r]$, we get
 $$\frac{\m r }{8\M} 
 \leq u(r)^{1/n} = |B_r \cap F|^{1/n},$$
 and taking the power $n$ of both sides yields the lower density estimate.
The upper density estimate is obtained by applying an analogous argument, using $G = F\cup B_r(x_0)$ as a comparison set  for $x_0 \in \partial^* F$ and $r<r_0$ satisfying \eqref{lineup}.
\end{proof}

The following lemma is a classical argument showing that a set that is close to $K $ in $L^1$ and satisfies uniform density estimates is close to $K$ in an $L^{\infty}$ sense.
\begin{lemma}\label{bound} Suppose that $F$ satisfies uniform density estimates as in \eqref{UDES}. Then there exists $C$ depending on $\m/\M$, $n$, and $\Lambda$ such that
$$ {\rm{hd}}(\pa F , \pa K)^n \leq C |F\Delta K|,$$
where ${\rm{hd}}(\cdot, \cdot) $
is the Hausdorff distance between sets.
In particular, for any $\eta>0$, there exists $\e>0$ such that if 
$|F\Delta K | < \e$, then $K_{1-\eta} \subset F \subset K_{1+\eta},$
where $K_{a} = aK.$
\end{lemma}

\begin{proof} 
Let $d = {\rm{hd}}(\pa F , \pa K).$ Then there is some $x \in \pa F$ such that either $B_d(x) $ is contained entirely in the complement of $K$ or $B_d(x)$ is entirely contained in $K$. If the first holds, then the lower density estimate in \eqref{UDES} implies that
$$ |F\Delta K| \geq |F\cap B_d(x) | \geq \frac{c_0\m^n}{\M^n} d^n,$$
while if the second holds, then the upper density estimate in  \eqref{UDES} implies that 
$$|F \Delta K | \geq |B_d(x) \setminus F| \geq \frac{c_0\m^n}{\M^n} d^n.$$ 
\end{proof}

We will make use of the following form of the Wulff inequality without a volume constraint.
\begin{lemma}\label{isop} Let $R_0>$ diam$(K)$ and $\Lambda >n$. Up to translation, the Wulff shape $K$ is the unique minimizer of the functional
$$\PP(F) + \Lambda\big| |F| - |K|\big| $$
among all sets $F\subset B_{R_0}$. 
\end{lemma}

\begin{proof}
Let $E$ be a minimizer of $\PP(F) + \Lambda\big| |F| - |K|\big| $ among all sets of finite perimeter $F\subset B_{R_0}$;  this functional is lower semicontinuous so such a set exists. Comparing with $K$, we find that
\begin{equation}\label{compare no vc}
 \PP (E) + \Lambda \big| |E| - |K|\big| \leq \PP(K) = n|K|.
 \end{equation}
The Wulff inequality implies that $|E| \leq |K|$, and so $\PP(E) \geq n|E|^{1/n'} |K|^{1/n} \geq n|E|$.
Thus \eqref{compare no vc} implies that $\Lambda \left( |K| - |E|\right) \leq n \left(|K| - |E|\right).$ Since $\Lambda >n$,
 it follows that $|E| = |K|$. It follows that $E$ must be a translation of $K$, the unique (up to translation) equality case in the Wulff inequality.
 \end{proof}

We are now ready to prove Theorem~\ref{th1}.  
\begin{proof}[Proof of Theorem~\ref{th1}]
By \eqref{FiMPStatement}, we need only to show that there exists a constant $C= C(n)$ such that 
\begin{equation}\label{new1}\BPP(E)^{4n/(n+1)} \leq C\delta_{\PP}(E),
\end{equation}
for any 
set of finite perimeter $E$ with $0<|E|<\infty.$
By Lemma~\ref{compact}, it suffices to consider sets contained in $B_{R_0}$.
Let us introduce the set 
$$\mathcal{F}_{\NN} = \Big\{f \  : \ \frac{\M}{\m} \leq \mathcal{N} \Big\}$$ 
for $\NN \geq 1$, recalling $\M$ and $\m$ defined in \eqref{mM}.
In Steps $1$-$4$, we prove
 that, for every $\NN \geq 1$, there exists  a constant $C= C(n,\NN)$ such that  \eqref{new1} holds for any surface energy $\PP$ corresponding to a surface tension $f \in \mathcal{F}_{\NN}$. 
In Step $5$, we remove the dependence of the constant on $\NN$. \\

\noindent \textit{Step 1:} \textit{Set-up.} \\
Suppose for the sake of contradiction that (\ref{new1}) is false for some $\NN$.
We may then find a sequence of  sets $\{E_j\}$ with $E_j \subset B_{R_0}$  and a sequence of surface energies $\{\PP_{j}\}$, each $\PP_j$ with corresponding surface tension $f^j\in \mathcal{F}_{\NN}$, Wulff shape $K_j$, and support function $f_*^j$, such that the following holds:
\begin{align}  
|E_j | &= |K_j|=1, \nonumber \\
\PP_j(E_j) &- \PP_j(K_j)\to 0, \nonumber\\ 
\label{contra}\PP_j(E_j)& < \PP_j(K_j) +c_1 \beta_{\PP_j}(E_j)^{4n/(n+1)},
 \end{align}
 where $c_1 = c_1(\NN,n)$ is a constant to be chosen later in the proof.

 Each $f^j $ is in $ \mathcal{F}_{\NN}$ and is normalized to make $|K_j| =1$ implying that $\{f^j\}$ is locally uniformly bounded above, and hence, by convexity,  locally uniformly Lipschitz. By the Arzela-Ascoli theorem, up to a subsequence, $f^j \to f^{\infty}$ locally uniformly. 
 The uniform convergence ensures that this limit function $f^{\infty}$ is a surface tension in $\mathcal{F}_{\NN}$. 
 We denote the corresponding surface energy by $\PP_{\infty}$, Wulff shape by $K_{\infty}$ and support function by $f_*^{\infty}$. Note that $|K_{\infty}| =1.$ 

There exists $c(\NN)$ such that  
$\PP_j(E) \geq c(\NN)P(E)$ for any set of finite perimeter $E$, again thanks to $f^j \in\mathcal{F}_{\NN}$ and $|K_j|=1$.
Then, since $\PP_j(E_j) \to n$ (as $\PP_j(K_j)=n$), the perimeters are uniformly bounded. Furthermore, $E_j \subset B_{R_0}$, so up to a subsequence, $E_j\to E_{\infty}$ in $L^1$ with  $|E_{\infty}| =1$.

Proposition~\ref{convergence}(1) implies that $\PP_{\infty}(E_{\infty}) \leq 
{\lim} \ \PP_j(E_j)= n$, so by the Wulff inequality, $E_{\infty} = K_{\infty}$ up to translation.
Furthermore, Proposition~\ref{convergence}(2) then  ensures that 
${\lim}\ \g_{\PP_j}(E_j) = \g_{\PP_{\infty}}(K_{\infty}) =\frac{n}{n-1}, $
and therefore, by (\ref{betaform}),
$$\underset{j \to \infty}{\lim} \ \beta_{\PP_j}(E_j)^2 =\underset{j \to \infty}{\lim} \ \frac{1}{n} \left(\PP_j(E_j )  -(n-1) \g_{\PP_j}(E_j ) \right)= 0.$$
\\
\noindent\textit{Step 2:} \textit{Replace each $E_j$ with a minimizer $F_j$. }\\
As in \cite{fuscojulin11}, the idea is to replace each $E_j$ with a set $F_j$ 
for which we can say more about the regularity.  We let $\e_j = \beta_{\PP_j}(E_j)$
and let $F_j$ be a minimizer to the problem 
$$\text{min} \left\{Q_j(F)= \PP_j(F)
 + \frac{m_{\PP_j}}{8M_{\PP_j}} |\beta_{\PP_j}(F)^2 -\e_j^2|
 + \Lambda \big| |F| - 1 \big|  \ :\ F \subset B_{R_0} \right\}$$
for a fixed $\Lambda > 4n$. 
Lemma~\ref{LSC} ensures that such a minimizer exists. As before,  $\PP_j(F_j) \geq c(\mathcal{N})P(F_j)$. Pairing this with \eqref{Perimeter Bounds} provides a uniform bound on $P(F_j)$, so by compactness,
 $F_j \to F_{\infty} $ in $L^1$ up to a subsequence for some $F_{\infty} \subset B_{R_0}$. 
 
 For each $j$, we use the fact that $F_j$ minimizes $Q_j$, choosing $E_j$ as a comparison set. This, combined with (\ref{contra}) and Lemma~\ref{isop}, yields
\begin{align}
\PP_j(F_j)  &
 + \frac{1}{8\NN} |\beta_{\PP_j}(F_j)^2 - \e_j^2| 
 + \Lambda \big| |F_j| - 1 \big|
 \leq Q_j(F_j) \leq \PP_j(E_j)\nonumber \\
 \label{fine}
&\leq \PP_j(K_j) + c_1 \e_j^{4n/(n+1)} \leq
\PP_j(F_j) + \Lambda \big| |F_j| - 1 \big| + c_1 \e_j^{4n/(n+1)}.
\end{align}
It follows that 
$\frac{1}{8\NN} \big|\beta_{\PP_j}(F_j)^2 -\e_j^2 \big| \leq c_1 \e_j^{4n/(n+1)}$, immediately
 implying that $\beta_{\PP_j}(F_j) \to 0$. 
Moreover, rearranging and using the fact that $\e_j \to 0$ 
and $\frac{4n}{n+1}>2$, we have
$$\frac{\e_j^2}{2^{(n+1)/2n} }\leq \e_j^2 - 8\NN c_1\e_j^{4n/(n+1)} \leq \beta_{\PP_j}(F_j)^2,$$
 where the exponent $(n+1)/2n$ is chosen so that, taking the power $2n/(n+1),$ we obtain
\begin{equation}\label{B}
\e_j^{4n/(n+1)} \leq 2\beta_{\PP_j}(F_j)^{4n/(n+1)}.
\end{equation}
In the last inequality in \eqref{fine}, if we replace $F_j$ with arbitrary set of finite perimeter $E\subset B_{R_0}$, then we obtain
$$\PP_j(F_j)+ \Lambda\big| |F_j| -1\big| \leq \PP_j(E) + \Lambda\big| |E| - 1\big| + c_1 \e_j^{4n/(n+1)},$$
again using Lemma~\ref{isop}. Taking the limit inferior as $j \to \infty$, this implies that $F_{\infty}$ is a minimizer of the problem 
$$ \min \left\{ \PP_{\infty}(F) + \Lambda||F|-1| \ :\ F \subset B_{R_0}\right\},$$
and so $F_{\infty } = K_{\infty}$ up to a translation by Lemma~\ref{isop}. With no loss of generality, we translate each $F_j$ such that $\inf\{ |(F_j+z) \Delta K_{\infty} | :z\in\Rn\}  = |F_j \Delta K_{\infty}|.$\\

\noindent\textit{Step 3:} \textit{ For $j$ sufficiently large,  $\frac{1}{2}K_j \subset F_j \subset 2K_j$ and $|F_j | = 1$.}\\ 
Lemma~\ref{UDE} implies that each $F_j$ satisfies uniform density estimates, and thus for $j$ sufficiently large, Lemma~\ref{bound} ensures that $\frac{1}{2} K_j \subset F_j\subset 2 K_j,$ as $|K_j \Delta F_j|  \leq |K_j \Delta K_{\infty}| + |K_{\infty} \Delta F_j|$
 and both terms on the right hand side go to zero.
 
Let $r_j>0$ be such that  $|r_jF_j| =1$. We may take $r_jF_j$ as a comparison set for $F_j$;  $r_j\leq2$ by Lemma~\ref{LSC}, so $r_jF_j \subset 4K_j\subset B_{R_0}$ as long as $R_0> 4\M > C\mathcal{N},$ the second inequality following from $|K_j|=1$. Since $\beta_{\PP_j}$ is invariant under scaling, $Q_j(F_j)\leq Q_j(r_jF_j)$ yields
\begin{align}\label{scale1} \PP_j(F_j)+ \Lambda|1- |F_j||\leq  r_j^{n-1} \PP_j(F_j) .\end{align}
This immediately implies that $r_j \geq 1$ for all $j$, in other words, $|F_j| \leq 1$.
Furthermore,  $r_j \to 1$ because $F_j \to K_{\infty}$ in $L^1$ and $|K_{\infty}|=1$. 
Suppose that, for some subsequence, $r_j>1$. Then, using $|F_j| = 1/r_j^{n}$,  \eqref{scale1} implies
\begin{equation}\label{etasmall}
\Lambda \leq \bigg(\frac{r_j^n(r_j^{n-1} - 1)}{r_j^n -1 } \bigg)\PP_j(F_j).
\end{equation}
For any $0<\eta<\frac{1}{n}$ and for $j$ sufficiently large, the right hand side is bounded by $ (1-\eta)\PP_j(F_j)$, as 
$\underset{r \to 1^+}{\lim}\ \frac{r^n(r^{n-1}-1)}{r^n -1} = \frac{n-1}{n}$.
Furthermore, $\PP_j(F_j) \leq n + \e_j^2$ since $Q_j(F_j) \leq Q_j(K_j)$, so \eqref{etasmall} implies that
$$\Lambda  \leq(1-\eta) \PP_j(F_j) \leq (1-\eta) \left(n + \e_j^2\right) \leq n
$$
for $j$ sufficiently large. Since $n<\Lambda$, we reach a contradiction, concluding that $|F_j| = 1$ for $j$ sufficiently large.
\\

\noindent \textit{Step 4:} \textit{Derive a contradiction to \eqref{contra}.}\\
We will show that $\beta_{\PP_j}(F_j)^{4n/(n+1)} \leq C \delta_{\PP_j}(F_j)$, which in turn will be used to contradict \eqref{contra}.
Adding and subtracting the term $\PP_{j}(K_j)/n = (n-1) \g_{\PP_j}(K_j)/n$ to \eqref{betaform}, we have
\begin{equation*}
\begin{split}
\beta_{\PP_j}(F_j)^2 &
\leq  \frac{ \PP_j(F_j)}{n} - \frac{n-1}{n} \int_{F_j} \frac{dx}{f^j_*(x)} =\delta_{\PP_j}(F_j) + \frac{n-1}{n}\bigg(\int_{K_j} \frac{dx}{f^j_*(x) } - \int_{F_j} \frac{dx}{f^j_*(x)}\bigg) \\
&=\delta_{\PP_j}(F_j) + \frac{n-1}{n}\bigg( \int_{F_j\setminus K_j} 1 - \frac{1}{f^j_*(x) }  \ dx + \int_{K_j\setminus F_j} \frac{1}{f^j_*(x) } - 1 \ dx 
\bigg).
\end{split}
  \end{equation*}
We now control the last term in terms of $\delta_{\PP_j}(F_j).$ Note the following: since
 $\frac{1}{2} K_j \subset F_j \subset 2K_j$,  the last term above is bounded by $C|F_j\Delta K_j| \leq \delta_{\PP_j}(F_j)^{1/2}.$ This could establish \eqref{new1} with the exponent $4$. However, with the following argument, we obtain the improved exponent $4n/(n+1)$.

As noted before, Lemma~\ref{UDE} implies that
each $F_j$ satisfies uniform density estimates (\ref{UDES}) 
with $m_{\PP_j}/M_{\PP_j}\geq 1/{\NN}$. 
The lower density estimate provides information about how far $f^j_*(x)$ can deviate from $1$ for $x \in F_j\setminus K_j$, thus bounding the first integrand. Indeed, arguing as in the proof of Lemma~\ref{bound}, for any $x \in F_j \setminus K_j$, let $d = f^j_*(x) -1$. 
The intersection $K_j\cap B_{d} (x) $ is empty by the definition of $f_*^j$,
 and thus $F_j\cap B_{d}(x) \subset F_j \setminus K_j$.
 Therefore, for $x \in \partial^*F_j \setminus K_j$, 
  $$\frac{c_0}{ {\NN}^n} d^n \leq |B_{d} (x) \cap F_j |
  \leq |F_j \Delta K_j| \leq C\delta_{\PP_j}(F_j)^{1/2}$$
 by the lower density estimate in \eqref{UDES} and the quantitative Wulff inequality as in \eqref{FiMPStatement}. In fact, this bound holds for any 
$x \in F_j \setminus  K_j$; since $F_j$ is bounded, for any $x \in F_j \setminus K_j$, there is some $y \in \partial^*F_j \setminus K_j$ such that $f_*^j(x) \leq f^*_j (y).$ Therefore, $f^j_*(x) - 1 \leq C\delta_{\PP_j}(F_j)^{1/2n}$ for all $x \in F_j \setminus K_j$, and so
\begin{equation} \label{aa}
\begin{split}\int_{F_j \setminus K_j } 1- \frac{1}{f^j_*(x) }dx &\leq \int_{F_j \setminus K_j} f_*(x) - 1 \ dx 
 \leq  \int_{F_j \setminus K_j} C\delta_{\PP_j}(F_j)^{1/2n} \ dx\\
 & = C |F_j\Delta K_j | \delta_{\PP_j}(F_j)^{1/2n} \leq  C \delta_{\PP_j}(F_j)^{1/2+1/2n},
\end{split}
\end{equation}
where $C = C({\NN},n)$ and the final inequality uses \eqref{FiMPStatement} once more.
The analogous argument using the upper density estimate in \eqref{UDES}, paired with the fact that eventually $\frac{1}{2}K_j \subset F_j$, 
provides an upper bound for the size of
 $1-f^j_*(x) $ for $x \in K_j\setminus F_j$, giving
\begin{equation}\label{bb}
\begin{split}
\int_{K_j\setminus F_j} \frac{1}{f^j_*(x)} - 1 \ dx \leq 2 \int_{K_j \setminus F_j }1 -f^j_*(x) \ dx
 \leq  C\delta_{\PP_j}(F_j)^{1/2 + 1/2n}.
\end{split}
\end{equation}
Combining (\ref{aa}) and (\ref{bb}), we conclude that 
\begin{equation}\label{feb}\beta_{\PP_j}(F_j)^{4n/(n+1)} \leq C_1\delta_{\PP_j}(F_j)\end{equation}
where $C_1= C_1 (\NN, n)$.

We now use the minimality of $F_j$, 
comparing against $E_j$, along with (\ref{contra}) and (\ref{B}) to obtain
\begin{align*}
\PP_j(F_j) & \leq \PP_j(E_j) 
\leq \PP_j(K_j) + c_1 \e_j^{4n/(n+1)} \leq  \PP_j(K_j) +2c_1 \beta_{\PP_j}(F_j)^{4n/(n+1)}.
 \end{align*}
By \eqref{B}, $\beta_{\PP_j}(F_j)$ is positive, so by choosing $c_1< n/2C_1$, this contradicts \eqref{feb}, thus proving \eqref{new1} for the class $\mathcal{F}_\NN$ with the constant $C$ depending on $n$ and $\NN$.\\
 
\noindent\textit{Step 5:} \textit{Remove the dependence on ${\NN}$ of the constant in (\ref{new1}).}\\
We argue as in \cite{FiMP10}. We will use the following notation: $\PP_{K}$ is the surface energy with Wulff shape $K$, surface tension $f^{K}$, and support function $f_*^{K}$. We use $\delta_K, \beta_K$, and $\g_K$ to denote $\delta_{\PP_K}, \beta_{\PP_K}$, and $\g_{\PP_K}$ respectively.

By John's Lemma (\cite[Theorem III]{John48}), for any convex set $K \subset \Rn$, there exists an affine transformation $L$ such that $\det L >0$ and  $B_1 \subset L(K) \subset B_n.$ This implies that $M_{L(K)} / m_{L(K)} \leq n$ and so $f^{L(K)} \in \mathcal{F}_n$.  Our goal is therefore to show that $\beta_{\PP}(E)$ and $\delta_{\PP}(E)$ are invariant under affine transformations. Indeed, once we verify that $\beta_K(E) = \beta_{L(K)}(L(E))$ and $ \delta_{K}(E) = \delta_{L(K)}(L(E)),$ we have 
$$\beta_K(E)^{4n/(n+1)} = \beta_{L(K)}(L(E))^{4n/(n+1)} \leq C(n) \delta_{L(K)}(L(E)) = 
C(n)  \delta_{K}(E) ,$$
and \eqref{new1} is proven with a constant depending only on $n$.

Suppose $E$ is a smooth, open, bounded set. Then 
$$
\PP_K(E) =
\underset{\e \to 0 }{\lim}\ \frac{ |E + \e K | - |E| }{\e};
$$
this is shown by applying the anisotropic coarea formula to the function 
$$
d^K(x, \pa E) :=\begin{cases} 
\inf \{ f_*(x-y) : y \in \pa E\} & {\rm if }\ x \in E^c\\
-\inf \{ f_*(x-y) : y \in \pa E\} & {\rm if }\ x \in E
\end{cases} 
$$
 and noting that 
$(E+\e K) \setminus E = \{ x: 0\leq  d^K(x, \pa E) <\e\}.$

Since $L$ is affine, $ |L(E + \e K)| - |L(E)| = \det L \left(|E + \e K | - |E|\right)$, and so 
$$\PP_K(E) = \underset{\e \to 0}{\lim} \ \frac{ |L(E + \e K)| - |L(E)|}{\e \det L}
=\frac{\PP_{L(K)}(L(E))}{\det L}.
$$
Since $|E| =|L(E)|/\det L$, we have
\begin{align*}
\delta_K(E)   = 
\frac{ \PP_K(E)}{n|K|^{1/n}|E|^{1/n'} } -1 
& = \frac{ \PP_{L(K)}(L(E))}{n|L(K)|^{1/n} |L(E)|^{1/n'} } -1 = \delta_{L(K)}(L(E)), 
\end{align*}
and thus $\delta_K(E)$ is invariant. Similarly,
\begin{align*}
f_*^K(L^{-1}z - y )& = 
\inf \Big\{ \lambda \ : \ \frac{L^{-1}(z) -y }{\lambda} \in K \Big\} 
= \inf \Big\{ \lambda \ : \ \frac{z - L(y)}{\lambda} \in L(K) \Big\} = f_*^{L(K)}\left(z-L(y)\right),
\end{align*} 
and thus
$$ \int_E \frac{dx}{f_*^K(x-y)}= \int_{L(E) } \frac{dz}{f_*^K\left(L^{-1}(z) - y\right) \det L}  = \int_{L(E)} \frac{dz}{f_*^{L(K)}\left(z-L(y)\right)\det L}.$$
Taking the supremum over $y\in \Rn$ of both sides, we have 
$$ \g_{K} (E)= \frac{\g_{L(K)}(L(E))}{\det L} .$$
From $(\ref{betaform})$,
$$
\beta_K(E) =\bigg(\frac{\PP_K(E) -(n-1) \g_K(E)}{n|K|^{1/n} |E|^{1/n'} } \bigg)^{1/2} .$$
We have just shown that, for the denominator, 
$$\bigg(\frac{1}{n|K|^{1/n} |E|^{1/n'} } \bigg)^{1/2} =  \bigg(\frac{\det L }{n|L(K)|^{1/n} |L(E)|^{1/n' }}\bigg)^{1/2}
,$$
and for the numerator, 
$$\bigg(\PP_K(E) -(n-1) \g_K(E) \bigg)^{1/2}  =\bigg( \frac{ \PP_{L(K) } (L(E)) - (n-1)\g_{L(K)}(L(E))}{\det L } \bigg)^{1/2}.$$
The term $\det L$ cancels, yielding
$$\beta_K(E) = \bigg(\frac{\PP_{L(K) } (L(E)) - (n-1)\g_{L(K)}}{n|L(K)|^{1/n} |L(E)|^{1/n'} }\bigg)^{1/2}
=\beta_{L(K)}(L(E)),
$$
showing that $\beta_K(E)$ too is invariant.
\end{proof}


\section{The case of $\lambda$-elliptic Surface Tension}\label{smoothsection}
In this section, we prove Theorem~\ref{Smooth}. 
This proof closely follows the proof of $(\ref{FuscoJulin})$ in \cite{fuscojulin11}. Using a selection principle argument and the regularity theory for $(\Lambda, r_0)$-minimizers of $\PP$, we reduce to the case of sets that are small $C^1$ perturbations of the Wulff shape $K$. In \cite{fuscojulin11}, this argument brings Fusco and Julin to the case of nearly spherical sets, at which point they call upon $(\ref{fugst})$, where Fuglede proved precisely this case in \cite{fuglede1989}.

We therefore prove in Proposition~\ref{Fug-type} an analogue of $(\ref{fugst})$ in the case of the anisotropic surface energy $\PP$ when $f$ is a $\lambda$-elliptic surface tension. The following lemma shows that if $E$ is a small $C^1$ perturbation of the Wulff shape $K$ with $|E| =|K|$, then the Taylor expansion of the surface energy vanishes at first order and takes the form \eqref{theexpansion}. We then use the quantitative Wulff inequality as in $(\ref{FiMPStatement})$ and the barycenter constraint along with \eqref{theexpansion} to prove Proposition~\ref{Fug-type}.

\begin{lemma} \label{expansionlem}
Suppose that $\PP$ is a surface energy corresponding to a $\lambda$-elliptic surface tension $f$, and $E$ is a set such that 
$|E| = |K|$ and
 $$\partial E = \{ x + u(x) \nu_K(x) \, :\, x \in \partial K \}$$ 
 where $u: \pa K \to \mathbb{R}$ and $\| u \|_{C^1(\partial K)} = \e.$
 There  exists $\e_0>0$ depending on $\lambda$ and $n$ such that if $\e<\e_0$, 
  \begin{align}\label{theexpansion}
 \PP(E) = \PP(K) + &\frac{1}{2}\int_{\partial K }(\nabla u)^{\rm{T}} \nabla^2 f(\nu_K)\nabla u \,d\Hn
 - \frac{1}{2} \int_{\partial K }H_K u^2 \,d\Hn
+\e \ O (\|u\|_{H^1(\partial K)}^2),
\end{align}
where $H_K$ is the mean curvature of $K$ and all derivatives are restricted to the tangential directions.
\end{lemma}
\begin{remark}
\rm{The second fundamental form $\rm{II}_K$ of $K$ satisfies
$$ \nabla^2 f(\nu_K(x)) \rm{II}_K(x) = \rm{Id}_{T_x\pa K}  \ \ \text{ for all } x \in \pa K.
$$
Therefore, $H_K= \text{tr}({\rm{II}_K})$ is equal to $\text{tr}(\nabla^2 f\, \rm{II}_K^2)$ and thus \eqref{theexpansion} agrees with, for example, \cite[Corollary $4.2$]{clarenz2004surfaces}.}
\end{remark}
\begin{proof}[Proof of Lemma~\ref{expansionlem}] 
For a point $x\in \partial K$, let $\{\tau_1, \hdots, \tau_{n-1}\}$ be normalized eigenvectors of $\nabla \nu_K$, where each $\tau_i $ corresponds to the eigenvalue $\lambda_i$.
  This set is an orthonormal basis for $T_xK$, and thus $\{\tau_1,\hdots, \tau_{n-1}, \nu_K\}$ is an orthonormal basis for $\Rn.$
A basis for $T_{x+u\nu_K}E$ is given by the set $\{g_1 , \hdots, g_{n-1}\}$, where,  adopting the notation $u_i = \partial_{\tau_i}u,$ 
$$g_i = \partial_{\tau_i}[x + u\nu_K] = (1+\lambda_i u)\tau_i + u_i \nu_K.$$

We make the standard identification of an $(n-1)$-vector with a vector in $\Rn$ in the following way. The norm of an $(n-1)$-vector $v_1\wedge\cdots\wedge v_{n-1}$ is given by $|v_1 \wedge \hdots \wedge v_{n-1}| = | \det (v_1, \hdots , v_{n-1})|.$ If $|v_1 \wedge \hdots \wedge v_{n-1}| \neq 0$, then the vectors $v_1,\hdots , v_{n-1} $ are linearly independent and we may consider the $n-1$ dimensional hyperplane $\Pi$ spanned by $v_1, \cdots , v_{n-1}.$ Letting $\nu$ be a normal vector to  $\Pi$, we make the identification 
$$v_1 \wedge \hdots \wedge v_{n-1} = \pm |v_1 \wedge \hdots \wedge v_{n-1}|\, \nu,$$
where the sign is chosen such that 
$\det (v_1, \hdots , v_{n-1} , \pm\nu) >0.$
In particular, we make the identifications
$$\tau_1 \wedge \hdots \wedge \tau_{n-1} = \nu_K, \qquad\frac{ g_1 \wedge \hdots \wedge g_{n-1}}{|g_1 \wedge \hdots \wedge g_{n-1}|} = \nu_E
, \quad \text{and} \quad\tau_1 \wedge \hdots \wedge \nu_K \wedge \hdots \wedge \tau_{n-1}  = - \tau_i.$$
The sign is negative in the third identification because 
\begin{align*}
 \det (\tau_1 ,\hdots, \nu_K , \hdots, \tau_{n-1}, -\tau_{i})
= -\det (\tau_1 ,\hdots,-\tau_{i},\hdots, \tau_{n-1}, \nu_K ) =\det (\tau_1 ,\hdots,\tau_{i},\hdots, \tau_{n-1}, \nu_K ) = 1.
\end{align*}
We let $w:=g_1 \wedge \hdots \wedge g_{n-1}$, and so
\begin{align}
w &= [ (1+\lambda_1 u )\tau_i + u_1 \nu_K] \wedge \hdots \wedge [(1+\lambda_{n-1}u )\tau_{n-1} + u_{n-1} \nu_K] \nonumber \\ 
& = \prod_{i=1}^{n-1} (1+ \lambda_i u) \nu_K - \sum_{i=1}^{n-1}  u_i \prod_{i\neq j} (1+ \lambda_ju)\tau_i \nonumber \\
&\label{ggg}   = \Big[1+ H_K u + \sum_{i<j } \lambda_i \lambda_j u^2 \Big]\nu_K
     -\sum_{i=1}^{n-1}u_i \Big[ 1  + \sum_{j\neq i } \lambda_j u\Big] \tau_i +\e\, O( |u|^2 + |\nabla u|^2).
\end{align}

In order to show (\ref{theexpansion}), the volume constraint is used to show that the first order terms in the Taylor expansion of the surface tension vanish.
We achieve this by expanding the volume in two different ways. 

The divergence theorem implies that 
\begin{align*} 
n|E|  = \int_{\partial E} x \cdot \nu_E\, d\Hn& =
 \int_{\partial K} (x+ u\, \nu_K) \cdot \left(\frac{w}{|w|}\right)|w|\,d\Hn = \int_{\partial K} (x+ u\,\nu_K) \cdot w\,d\Hn. \end{align*}
Adding and subtracting $\nu_K = \tau_1 \wedge \hdots \wedge \tau_{n-1}$, and using \eqref{ggg} and the fact that $\nu_K \cdot \tau_i = 0$, we have \begin{align*}
n|E|  =   \int_{\partial K } x\cdot \nu_K\,  d\Hn
+\int_{\partial K} u   + x \cdot (w &- \nu_K )+H_K u^2\, d\Hn +\e \, O( \|u\|_{H^1(\partial K)}^2
).
\end{align*}
Since $ \int_{\partial K } x\cdot \nu_K  \,d\Hn = n|K|,$
 the volume constraint $|E| = |K|$ implies that 
\begin{equation}\label{vol1}
\begin{split} \int_{\partial K} x \cdot (w &- \nu_K )\, d\Hn= -\int_{\partial K} u +H_K u^2 \, d\Hn+\e \, O( \|u\|_{H^1(\partial K)}^2).
\end{split}
\end{equation}

Now we expand the volume in a different way. Because $f$ is a $\lambda$-elliptic surface tension, 
the Wulff shape $K$ is $C^2$ with mean curvature depending on $\lambda$ and $n$. 
Therefore, there exists $t_0=t_0(\lambda, n)>0$ such that the neighborhood 
$$D = \{ x+t\nu_K(x) :x\in \partial K , t\in (-t_0, t_0) \}$$
satisfies the following property: for each $y\in D,$ there is a unique projection $\pi:D \to \partial K$ such that
$\pi(y) =x$ if and only if $y = x + t\nu_K(x)$ for some $t\in (-t_0, t_0)$. In this way, we extend the normal vector field $\nu_K$ to a vector field $N_K$ defined on $D$ by letting $N_K:D\to\Rn$ be defined by $N_K(y)= \nu_K(\pi(y)).$ We also extend $u$ to be defined on $D$ by letting $u(y) = u(\pi(y))$ for all $y \in D$.
Therefore, if $\e_0<t_0$, $\partial E$ may be realized as the time $t=1$ image of $\partial K$ under the flow defined by
$$\frac{d}{dt} \psi_t(x) = uN_K(\psi_t(x)), \qquad \psi_0(x) = x.$$
Such a flow is given by 
$\psi_t(x) = x + tuN_K,$
and so 
$\nabla \psi_t(x) = {\rm{Id}} + tA$ where $A= \nabla (uN_K).$
A small adaptation of the proof of \cite[Lemma $17.4$]{maggi2012sets} gives
\begin{equation}\label{Jacobian}
J\psi_t = 1 + t \text{ tr}(A) +\frac{t^2}{2} (\text{tr}(A)^2 - \text{tr}(A^2)) + \e \, O(|u|^2 + |\nabla u|^2).
\end{equation}
Integrating by parts, it is easily verified that
\begin{align*}
\int_{K} \text{tr} (A)^2 - \text{tr} (A^2)\, dx
& = \int_{K } \DIV (uN_K\,\DIV(uN_K))\, dx - \int_{\partial K } \sum_{i,j=1}^n 
(uN_K)^{(i)}\partial_i(uN_K)^{(j)} \nu_K^{(j)}\,d\Hn 
\\
&  = \int_{K} \text{div} (uN_K\,\DIV(uN_K)) \,dx - \int_{\partial K }
u \nabla u \cdot \nu_K\, d\Hn .
\end{align*}
The second equality is clear by choosing the basis $\tau_1, \hdots, \tau_{n-1}, \tau_{n}$, where $\tau_n = \nu_K$.
Furthermore, the divergence theorem implies that
$$ \int_{K} \DIV(uN_K\, \DIV(uN_K))\, dx = \int_{\partial K } u \ \DIV(uN_K) \,  d\Hn = 
\int_{\partial K}u \nabla u \cdot \nu_K + H_K u^2 \,d\Hn,$$
so that
$$
\int_{K} \text{tr} (A)^2 - \text{tr} (A^2)\, dx=
\int_{\partial K} H_K u^2 \,d\Hn.
$$
With this and \eqref{Jacobian} in hand, we have the following expansion of the volume:
\begin{align*}
|\psi_t(K) | = \int_K J\psi_t \,dx
&= |K| + t \int_{\partial K } u \,d\Hn 
+ \frac{t^2}{2} \int_{\partial K} H_K  u^2\,d\Hn + t^3 \e \, O(\|u\|_{H^1(\partial K)}^2).
\end{align*}
Therefore, the volume constraint $|K| =|E|= |\psi_1(K)|$ implies that 
\begin{equation}\label{vol2} \int_{\partial K } u \, d\Hn=- \frac{1}{2}\int_{\partial K} H_K u^2\,  d\Hn+  \e \,O(\|u\|_{H^1(\partial K)}^2).
\end{equation} 
 Combining $(\ref{vol1})$ and $(\ref{vol2})$, we conclude that 
 \begin{equation}\label{volume}
  \int_{\partial K} x \cdot (w-\nu_K )\, d\Hn
  =-\frac{1}{2}\int_{\partial K } u^2H_K\, d\Hn + \e \, O(\|u\|_{H^1(\partial K)}^2).
 \end{equation} 

We now proceed with 
a Taylor expansion of the surface energy of $E$:
\begin{align*} 
\PP(E)  =& \int_{\partial^* E} f(\nu_E) \,d\Hn=
  \int_{\partial K} f\Big(\frac{w}{|w|}\Big) |w|\,d\Hn= \int_{\partial K} f(w)\,d\Hn \\
  = &\int_{\partial K } f(\nu_K)\, d\Hn
  + \int_{\partial K } \nabla f(\nu_K)\cdot (w - \nu_K )\,d\Hn
\\
&+ \frac{1}{2}\int_{\partial K } [w-\nu_K ]^{\text{T}}  \nabla^2 f(\nu_K)  [w - \nu_K ] \,d\Hn +\e \, O( \|u\|_{H^1(\partial K)}^2),
\end{align*}
so, recalling that $\nabla f(\nu_K(x)) = x$ by \eqref{gradf},
$$\PP(E) =\PP(K)
+ \int_{\partial K }x\cdot (w- \nu_K )\,d\Hn
+\frac{1}{2}\int_{\partial K } \sum_{i,j = 1}^{n-1} u_iu_j(\tau_i^{\text{T}} \nabla^2 f(\nu_K) \tau_j) \,d\Hn
+ \e \, O( \|u\|_{H^1(\partial K)}^2).
$$
Applying $(\ref{volume})$ yields $(\ref{theexpansion})$, completing the proof.
\end{proof}

We now prove Proposition~\ref{Fug-type}, using $(\ref{theexpansion})$ as a major tool.

\begin{proof}[Proof of Proposition~\ref{Fug-type}]
Suppose $E$ is a set as in the hypothesis of the proposition, 
i.e., $|E| = |K|,$ $\bary E = \bary K$, and
 $$\pa E = \{ x + u(x) \nu_K(x): x \in \pa K \},$$ 
where $u: \pa K \to \mathbb{R}$ is a function such that $u \in C^1(\pa K)$ and $\|u\|_{{C^1}(\pa K)}= \e \leq \e_1$ with $\e_1$ to be fixed during the proof. Up to multiplying $f$ by a constant, which changes $\lambda$ by the same factor and leaves $\m/\M$ unchanged, we may assume that $|K| =1$. Let 
$$B(u) = \frac{1}{2}\int_{\pa K }(\nabla u)^{\text{T}} \nabla^2 f(\nu_K)\nabla u\, d\Hn
 -\frac{1}{2} \int_{\pa K }H_K u^2 \,d\Hn,$$
so that, by \eqref{theexpansion}, 
\begin{equation}\label{deficit bilinear form}
\delta_{\PP}(E) = \frac{1}{n} B(u) +\e \,O (\|u\|_{H^1(\pa K)}^2)
\end{equation}
as long as $\e_1 \leq \e_0$ for $\e_0$ from Lemma~\ref{expansionlem}.

\vspace{3mm}
\noindent {\it{Step 1:}} {\it There exists $C = C(n,\lambda, \m/\M)$ such that, for $\e_1$ small enough depending on $ \m/\M$ and $\lambda$,}
\begin{equation}\label{L1squared} \Big(  \int_{\pa K} |u|\, d \Hn\Big)^2 \leq C\delta_{\PP}(E).
\end{equation}

\vspace{3mm}
\noindent {\it{Step 1(a):}} {\it There exists $C= C(n,\m/\M)$ such that, for $\e_1= \e_1(\m/\M)$ small enough,}
\begin{equation}\label{okok}|E \Delta K| \leq C\delta_{\PP}(E)^{1/2}.\end{equation}
The quantitative Wulff inequality in the form \eqref{FiMPStatement} states that
$ |E \Delta(K+ x_0)|\leq   C(n)\delta_{\PP}(E)^{1/2}$
for some $x_0 \in \Rn$, 
so by the triangle inequality,
 \begin{equation}\label{bound1}
 |E \Delta K|  \leq C(n)  \delta_{\PP}(E)^{1/2}  + |(K+x_0) \Delta K| .
 \end{equation}
 It therefore suffices to show that $|(K+x_0) \Delta K| \leq C\delta_{\PP}(E)^{1/2}.$ By \cite[Lemma $17.9$]{maggi2012sets},
 \begin{equation}\label{center} 
 |K\Delta (K+x_0)| \leq 2 |x_0| P(K) \leq\frac{2n}{\m} |x_0| .
 \end{equation}
 Furthermore,  the barycenter constraint $\bary E= \bary K$ implies that
  $$
  x_0 =  \int_{K} x_0\, dx  =  \int_{E} x\, dx - \int_K x - x_0\, dx
  = \int_{E} x\,dx - \int_{K + x_0} x\,dx.
$$
 For $\e_1$ small enough depending on $\M/\m$, $E, K+x_0 \subset B_{2\M}$, a fact that is verified geometrically  since $|x_0| \to 0$ as $\e \to 0$ and thus $|x_0|$ may be taken as small as needed. Therefore,
 \begin{align*}
 |x_0|   =\bigg| \int_{E} x\,dx - \int_{K + x_0} x\,dx\bigg|
  \leq 2\M |E \Delta (K+x_0)| \leq \M C(n)\delta_{\PP}(E)^{1/2},
\end{align*}
where the second inequality comes from \eqref{FiMPStatement}. 
This,  \eqref{center}, and \eqref{bound1} prove \eqref{okok}.\\

\noindent {\it{Step 1(b):}} { \it For $\e_1$ sufficiently small depending on $\lambda$ and $n$, 
 \begin{equation}\label{koko}   \int_{\pa K} | u| \,  d\Hn \leq 2  |E \Delta K |  .\end{equation}
 }
Let $d_K(x) = \text{dist} (x,\pa K)$. As in the proof of Lemma~\ref{expansionlem}, there exists  $t_0 = t_0(\lambda, n)$ such that for all $t<t_0$, 
 $\{ d_K = t\} = \{x + t\nu_K(x) \}$. Take $\e_1<t_0$ and let $G_{t} =\{d_K = t\} \cap (E\setminus K)$. Then
 \begin{eqnarray*}
 E\setminus K &=& \{x+ t \nu_K\,  : \, x\in \{ x \in \pa K : u(x)>0\},\,  t \in (0, u(x) )  \},\\
 G_{t}& =&  \{ x + t \nu_K : \, x\in \{ x \in \pa K : u(x) >t\}\}.
 \end{eqnarray*}
 The coarea formula and the area formula  imply that
 \begin{align*} 
 |E\setminus K | &= \int_{E\setminus K } |\nabla d_K| \,dx = \int_0^{\infty} dt \int_{G_{t}}\, d\Hn
  = \int_0^{\infty} \,dt \int_{\{u>t\}} J (\text{Id} + t\nu_K)\, d\Hn,
  \end{align*}
so
  \begin{align*}
 |E\setminus K |  & \geq \frac{1}{2} \int_0^{\infty} dt \int_{\{u>t\} } \,d\Hn  = \frac{1}{2} \int_0^{\infty} |\{u>t\} |\,dt
  = \frac{1}{2} \int_{\pa K}  u^+\, d \Hn.
 \end{align*}
The analogous argument yields $|K\setminus E|  \geq \frac{1}{2} \int_{\pa K} u^-\,d\Hn$, and \eqref{koko} is shown.
Combining (\ref{okok}) and (\ref{koko}) implies (\ref{L1squared}).

 \vspace{3mm}
\noindent {\it{Step 2:}}
 {\it There exists $C=C(n, \lambda,\m/\M)$ such that, for  $\e_1=\e_1(n, \lambda, \m/\M)$ small enough,}
\begin{equation}\label{L2bound1} \|u\|_{H^1(\pa K)}^2 \leq C \delta_{\PP}(u).
\end{equation}
The $\lambda$-ellipticity of $f$ implies
\begin{align*}
 \int_{\pa K} |\nabla u|^2\, d\Hn & 
 \leq 
\frac{1}{\lambda} \int_{\pa K} (\nabla u)^{\text{T}} \nabla^2 f(\nu_K ) (\nabla u)\, d\Hn  = 
\frac{1}{\lambda}\Big(2B(u) + \int_{\pa K} H_K  |u|^2\,d\Hn\Big).
\end{align*}
The Wulff shape $K$ is bounded and $C^2$, so $H_K$ is bounded by a constant $C= C(n, \lambda)$. Therefore,
\begin{equation}\label{bound on gradient}
 \int_{\pa K} |\nabla u|^2\, d\Hn   \leq 
 \frac{2}{\lambda}B(u) +C \int_{\pa K}  |u|^2\,d\Hn.
\end{equation}

As pointed out in \cite[proof of Theorem 4]{de2014sharp}, one may produce a version of Nash's inequality on $\pa K$ that takes the form
 \begin{align}\label{nash}
 \int_{\pa K} |u|^2\, d \Hn \leq c \eta^{(n+2)/n} \int_{\pa K} |\nabla u|^2 \,d\Hn + \frac{c}{\eta^{(n+2)/2}}\Big( \int_{\pa K} |u| \,d \Hn\Big)^2
 \end{align}
for all $\eta>0$, where $c$ is a constant depending on $H_{K}$ (and therefore on $\lambda$ and $n$) and $\M/\m$. Indeed, as in $\Rn$, this form of Nash's inequality is a consequence of the Sobolev inequality on $\pa K$, shown in \cite[Section 18]{simon1984lectures}. 
We pair \eqref{nash} with \eqref{bound on gradient} and  \eqref{L1squared} to obtain
\begin{align*}
\int_{\pa K} |\nabla u|^2\, d\Hn&   \leq 
 \frac{2}{\lambda}B(u) +
 C \eta^{(n+2)/n}  \int_{\pa K} |\nabla u|^2\, d\Hn + \frac{C}{\eta^{(n+2)/2}} \delta_{\PP}(E).
 \end{align*}
 For $\eta$ small enough, we absorb the middle term into the left hand side. Then, recalling \eqref{deficit bilinear form}, we have
\begin{align*}
\frac{1}{2}\int_{\pa K} |\nabla u|^2\, d\Hn  \leq 
C \delta_{\PP}(E) + \e\,O \big(\|u\|_{H^1(\pa K)}^2\big).
\end{align*}
Combining this estimate with \eqref{nash} and \eqref{L1squared}, we find that $\int_{\pa K} |u|^2 \, d\Hn$ is also bounded by $C \delta_{\PP}(E) + \e\,O \big(\|u\|_{H^1(\pa K)}^2\big)$. Therefore, 
\begin{align*}
\| u \|_{H^1(\pa K)}^2  \leq 
C \delta_{\PP}(E) + \e\,O \big(\|u\|_{H^1(\pa K)}^2\big).
\end{align*}
Finally, taking $\e_1$ small enough, we absorb the second term on the right, proving \eqref{L2bound1}.
\end{proof}

We now show that if $\partial E = \{ x+ u\nu_K: x \in \partial K\}$ with $\|u\|_{C^1(\pa K)}$ small, then $\BPP(E)$ is controlled by $\|u\|_{H^1(\partial K)}.$ With the notation from the proof of Lemma~\ref{expansionlem},
\begin{align*} 
n|K|\BPP(E)^2  \leq& \int_{\partial E} f(\nu_E) - \frac{x}{f_*(x)}\cdot \nu_{E} \, d\Hn = \int_{\partial K } f(w) - x\cdot w \,  d\Hn.
\end{align*}
From the expansion of $\PP$ in the proof of Lemma~\ref{expansionlem} and the fact that $x \cdot \nu_K = f(\nu_K)$  by \eqref{bdK}, the right hand side is equal to
\begin{align*}
 \frac{1}{2}\int_{\partial K } (\nabla u)^{\text{T}} \nabla^2 f(\nu_K) \nabla u\, d\Hn
+ \e \, O( \|u\|_{H^1(\partial K)}^2) \leq C \| u\|_{H^1(\pa K)}^2+ \e \, O( \|u\|_{H^1(\partial K)}^2),
\end{align*}
where $C= \| \nabla^2 f\|_{C^{0}(\pa K)}$.
For $\e$ sufficiently small, we absorb the term $\e \, O( \|u\|_{H^1(\partial K)}^2)$ and have
\begin{equation}\label{betaexp}
\BPP(E)^2 \leq \frac{C}{n|K|} \|u\|_{H^1(\partial K)}^2.
\end{equation}
\begin{remark}
\rm{
This is the first point at which we use the upper bound on the Hessian of $f$. In other words, Proposition~\ref{Fug-type} still holds for surface tensions $f \in C^{1,1}(\Rn \setminus \{0\})$ that satisfy the lower bound on the Hessian in the definition of $\lambda$-ellipticity. 
}
\end{remark}

Next, we prove Theorem~\ref{Smooth}, for which we need the following definition.
\begin{definition}\label{pm} A set of finite perimeter $E$ is a $(\Lambda, r_0)$-minimizer of $\PP$, for some $0\leq \Lambda <\infty $ and $r_0>0$,
if 
 $$\PP(E ; B(x,r) ) \leq \PP(F; B(x,r)) + \Lambda|E\Delta F|$$
 for $E\Delta F \subset \subset B(x,r) $ and $r<r_0$. 
 \end{definition}
 
\begin{proof}[Proof of Theorem~\ref{Smooth}]
 Proposition~\ref{poincare} implies that the proof reduces
 to showing
\begin{equation}\label{betaonly}\BPP(E)^2 \leq C\delta_{\PP}(E).
\end{equation}  
where $C = C(n ,\lambda, \| \nabla^2 f\|_{C^0(\pa K)} , \m/\M)$.
Suppose for contradiction that \eqref{betaonly} fails. There exists a sequence $\{E_j\}$ such that $|E_j| = |K|$ for all $j$, $\delta_{\PP}(E_j) \to 0$, and 
\begin{align}\label{contrastatement}
\PP(E_j) \leq \PP(K) + c_2 \BPP(E_j)^2
\end{align} 
for $c_2$ to be chosen at the end of this proof. Arguing as in the proof of Theorem~\ref{th1}, we determine that, up to a subsequence, $\{E_j\}$ converges in $L^1$ to a translation of $K$.
As in the proof of Theorem~\ref{th1} (and as in \cite{fuscojulin11}), we replace the sequence $\{E_j\}$ with
 a new sequence $\{F_j\}$, where each $F_j$ is a minimizer of the problem
$$\min \left\{ Q_j (E)= \PP(E) + \frac{|K|\m}{8\M} \left| \BPP(E)^2  - \e_j^2\right|+ \Lambda \big| |E | - |K|\big| 
  \ \ : \ E \subset B_{R_0} \right\}$$
 with $\e_j = \BPP(E_j)$; existence for this problem is shown in
 Lemma~\ref{LSC}.  
Continuing as in the proof of Theorem~\ref{th1}, we determine that
  \begin{equation}\label{epsbeta2}
  \e_j^2 \leq 2 \BPP(F_j)^2,
  \end{equation}
that up to a subsequence and translation, $F_j \to K$ in $L^1$, and 
  that $|F_j| = |K|$ for $j$ sufficiently large.
By Lemma~\ref{UDE}, each $F_j$ satisfies uniform density estimates, and so
by Lemma~\ref{bound}, for any $\eta>0$, we may choose $j$ sufficiently large such that 
$K_{1-\eta} \subset F_j \subset K_{1+\eta}.$

Arguing as in \cite{fuscojulin11}, we show that $F_j$ is a $(\Lambda, r_0)$-minimizer of $\PP$ for $j$ large enough, where $\Lambda$ and $r_0$ are uniform in $j$. Let $G$ such that $G\Delta F_j \subset \subset B_r(x_0)$
 for $x_0 \in F_j$ and for $r<r_0$, where $r_0$ is to be fixed during the proof.
 For any $\eta>0$, if $B_r(x_0) \subset K_{1-\eta}$, then trivially $\PP(G) \geq \PP(F_j)$. 
 If $B_r(x_0) \not\subset K_{1-\eta},$
  then for $\eta$ sufficiently small, Lemma~\ref{3.1} implies that 
   $|y_{F_j}| \leq 1/4 $ and $|y_G|\leq 1/4$. 
   Furthermore, by choosing $\eta$ and $r_0$ sufficiently small, we may take $B_r(x_0) \cap K_{1/2} = \emptyset.$
The minimality of  $F_j$ implies $Q(F_j) \leq Q(G)$; after rearranging and applying the triangle inequality, this implies that
\begin{align}\label{aboveok}
\PP(F_j)& \leq \PP(G) + \Lambda | F_j \Delta G |  + \frac{|K|\m}{8\M}\left| \BPP(G)^2 - \BPP(F_j)^2 \right|.
 \end{align}
As in \eqref{bbound} in the proof of Lemma~\ref{UDE},
\begin{equation*}
\frac{|K|\m}{8\M}\left| \BPP (F)^2  - \BPP(G)^2 \right| 
 \leq 
 \ \frac{\left| \PP(F_j) - \PP(G) \right|}{2} + \frac{\left| \g_{\PP}(F_j) - \g_{\PP}(G) \right|}{2}  +4|F_j \Delta G|
\end{equation*} 
for $r_0$ small enough depending on $n$.
If $\PP(F_j) \leq \PP(G)$, then the $(\Lambda, r_0)$-minimizer condition is automatically satisfied. Otherwise, subtracting $\frac{1}{2}\PP(F_j)$ from both sides of \eqref{aboveok} 
and renormalizing, we have
\begin{align}\label{january}
\PP(F_j) \leq \PP(G) +  |\g_{\PP}(G) - \g_{\PP}(F_j)|+(8+ 2\Lambda)|F_j \Delta G|.
\end{align}
To control $|\g_{\PP}(G) - \g_{\PP}(F_j)|$, we need something sharper than the H\"{o}lder modulus of continuity of $\g_{\PP}$ given in Proposition~\ref{convergence}(2). Indeed, $\g_{\PP}$ is Lipschitz continuous for sets whose intersection contains a ball around their centers:
$$
\g_{\PP}(F_j) - \g_{\PP}(G) 
\leq \int_{F_j} \frac{dx}{f_*(x-y_{F_j})} - \int_G \frac{dx}{f_*(x-y_{F_j})} 
= \int_{F_j \Delta G} \frac{dx}{f_*(x-y_{F_j})},
$$
and analogously, 
$$
\g_{\PP}(G) - \g_{\PP}(F_j ) \leq \int_{F_j \Delta G} \frac{dx}{f_*(x-y_G)}.
$$
Since $B_r \cap K_{1/2} = \emptyset$, $|y_{F_j}| \leq 1/4$, and $|y_G| \leq 1/4$, we know that $1/f_*(x-y_{F_j}) \geq 4/m_{\PP}$ and $1/f_*(x-y_G) \geq 4/m_{\PP}$ for any $x \in F_j \Delta G$, implying that 
$$|\g_{\PP}(F_j) - \g_{\PP}(G) | \leq \frac{4}{m_{\PP}}|F_j \Delta G|.$$
Therefore, \eqref{january} becomes
\begin{equation}
\PP(F_j) \leq \PP(G) + \Lambda_0 \left|F_j \Delta G\right|,
\end{equation}
where $\Lambda_0= 8 + 2\Lambda + 4/\m,$ and so $F_j$ is a $(\Lambda_0, r_0)$-minimizer for $j$ large enough. 

We now exploit some regularity theorems for sets $F_j$ that are $(\Lambda, r_0)$-minimizers that converge in $L^1$ to a $C^2$ set. First, let us introduce a bit of notation. 
For $x \in \Rn$, $r>0$, and $\nu \in S^{n-1}$, we define
\begin{align*} 
\textbf{C}_{\nu}(x,r) = \{ y\in \mathbb{R}^{n} : |p_{\nu} (y-x)|<r, |q_{\nu}(y-x) < r\},\\
\textbf{D}_{\nu} (x,r) = \{ y \in \Rn : |p_{\nu}(y-x)| <r , |q_{\nu}(y-x) | = 0\},
\end{align*}
where $q_{\nu}(y) =  y \cdot \nu$ and $p_{\nu}(y) = y - (y \cdot \nu) y.$
We then define the \textit{cylindrical excess} of $E$ at $x$ in direction $\nu$ at scale $r$ to be
$$\textbf{exc} (E, x ,r,\nu ) = \frac{1}{r^{n-1}} \int_{\textbf{C}_{\nu}(x,r) \cap \partial^* E}  \frac{|\nu_E - \nu|^2}{2} \,d\Hn$$

The following regularity theorem for almost minimizers of an elliptic integrand is the translation in the language of sets of finite perimeter of a classical result in the theory of currents, see \cite{alm66, SSA77, bomb82, DuzaarSteffen02}. For a closer statement to ours, see Lemma 3.1 in \cite{DePMag14}.

\begin{theorem} \label{epsreg}
Let $f$ be a $\lambda$-elliptic surface tension with corresponding surface energy $\PP$. Suppose $E$ is a $(\Lambda, r_0)$-minimizer  of $\PP$. For all $\alpha<1$ there exist constants $\e$ and $C_1$ depending on $n, \lambda$ and $\alpha$ such that if
$${\rm{\bf{exc}}}(E, x,r,\nu) + \Lambda r <\e$$
then there exists $u\in C^{1,\alpha}(\textbf{D}_{\nu}(x,r))$ with $u(x) =0$ such that 
\begin{align*}
 \textbf{C}_{\nu}(x,r/2) \cap \partial^*E & = ({\rm{Id}} + u \nu)({\rm{\bf{D}}}_{\nu}(x,r/2)),\\
  \|u\|_{C^0({\rm{\bf{D}}}_\nu(x_0, r/2))}& < C_1r\, {\rm{\bf{exc}}}(E,x,r,\nu)^{1/(2n-2)},\\
  \|\nabla u\|_{C^0({\rm{\bf{D}}}_{\nu}(x_0, r/2))}
&< C_1\, {\rm{\bf{exc}}}(E,x,r,\nu)^{1/(2n-2)},\\
 \text{and} \qquad r^{\alpha} [\nabla u ]_{C^{0,\alpha}({\rm{\bf{D}}}_{\nu}(x,r/2))}& < C_1\, {\rm{\bf{exc}}}(E,x,r,\nu)^{1/2}.
\end{align*} 
\end{theorem}
Applying Theorem~\ref{epsreg} as in \cite{CiLe12}, we come to prove the following statement.
 \begin{theorem}\label{improvedconvergence}
Let $f$ be $\lambda$-elliptic with corresponding surface energy $\PP$ and let $\{E_j\}$ be a sequence of $(\Lambda, r_0)$-minimizers such that $E_j \to E$ in $L^1$, with $\partial E \in C^2$. Then there exist functions $\psi_j \in C^1(\partial E)$ such that 
$$\partial E_j =({\rm{Id}} + \psi_j \nu_E)(\partial E),$$
and $\|\psi_j\|_{C^1(\partial E)} \to 0.$
\end{theorem} 

 Theorem~\ref{improvedconvergence} implies that we may express $\partial F_j$ as 
$$\partial F_j  = \{ x+ \psi_j \nu_K :  x \in \partial K\},$$
where $\|\psi_j\|_{C^1(\partial K)} \to 0$. 
Moreover, $\bary F_j = \bary K$ and $|F_j| = |K|$, so Proposition~\ref{Fug-type} and \eqref{betaexp} imply that
\begin{equation}\label{here} C \delta_{\PP}(F_j)\geq \| \psi_j \|_{H^1(\partial K)}^2 \geq c \BPP(F_j)^2.
\end{equation}
On the other hand, $F_j$ minimizes $Q_j$, so choosing $E_j$ as a comparison set and using (\ref{contrastatement}) and (\ref{epsbeta2}), we have 
$$\PP(F_j) \leq \PP(E_j) \leq \PP(K) + c_2\e_j^2 \leq \PP(K) + 2c_2\BPP(F_j)^2.$$
By \eqref{epsbeta2}, $\BPP(F_j)>0,$. Then, using \eqref{here} and choosing $c_2$ sufficiently small, we reach a contradiction. 
\end{proof}  


\section{The case of  crystalline surface tension in dimension $2$}\label{crystalsection}
In this section, we prove Theorem~\ref{dim2}.
As in the previous section, we begin by showing the result in a special case, and then use a selection principle argument paired with specific regularity properties to reduce to this case.

Let $n=2$ and suppose that $f$ is a crystalline surface tension as defined in Definition~\ref{crystals}, with $\PP$ the corresponding anisotropic surface energy.
 The corresponding Wulff shape $K \subset \mathbb{R}^2 $ 
is a convex polygon with normal vectors $\{\nu_i\}_{i=1}^N$. Let us fix some notation to describe $K$, illustrated in Figure $1$. Denote by $s_i$ the side of $K$
 with normal vector $\nu_i$, choosing the indices
 such that $s_i$ is adjacent to $s_{i+1}$ and $s_{i-1}$. Let $\theta_i \in (0,\pi)$ be the angle between $s_i$ and $s_{i+1}$, adopting the convention that $s_{n+1}  = s_1$. Let $H_i$ be the distance from the origin to the side $s_i$. By construction,
\begin{equation}\label{Hi}
f(\nu_i) = H_i.
\end{equation}

We say that a set $E\subset \mathbb{R}^2$ is \textit{parallel to $K$} if $E$ is an open convex polygon with $\{\nu_E\}= \{\nu_i\}_{i=1}^N$, that is, $\nu_E(x)\in \{\nu_i\}_{i=1}^N$ for all $ x \in \partial^*E$, and for each $i \in \{1,\hdots, N\}$, there exists $x \in \pa^*E$ with $\nu_E(x) = \nu_i$. For a set $E$ that is parallel to $K$, we denote by $\sigma_i$ the side of $E$ 
with normal vector $\nu_i$, and $h_i$ the distance between the origin and $\sigma_i$; again see Figure $1$.
 We define $\e_i= h_i - H_i$. Notice that $\e_i$ has a sign, with $\e_i \geq 0$ when dist$(0, s_i) \leq $ dist$(0, \sigma_i)$ and $\e_i \leq 0 $ 
when dist$(0, s_i) \geq $ dist$(0, \sigma_i)$. For simplicity of notation, we let $|s | = \mathcal{H}^1(s)$ for any line segment $s$. 

The following proposition proves strong form stability for sets $E$ that are parallel to $K$ such that  $|E|=|K|$ and $|E\Delta K| ={\inf} \{|E\Delta (K+y)| : y\in \mathbb{R}^2\}$. Then, by a selection principle-type argument and a rigidity result, we will reduce to this case.

\begin{proposition}
\label{crystalprop}
Let $E\subset \mathbb{R}^2$ be parallel to $K$ such that $|E| = |K|$ and $|E\Delta K| ={\inf} \{ |E\Delta (K+y)|: y\in \mathbb{R}^2\}$.
 Then there exists a constant $C$ depending on $f$ such that 
$$\BPP(E)^2 \leq C\delta_{\PP}(E).$$
\end{proposition}
\begin{proof}
Let $E$ be as in the hypothesis of the proposition. 
By (\ref{Hi}), we have
\begin{align*} 
\PP(E)  =\sum_{i=1}^N H_i |\sigma_i|,\qquad
\PP(K)  =\sum_{i=1}^N H_i |s_i|, \qquad
|E| =
 \sum_{i=1}^N \frac{h_i |\sigma_i|}{2}, \qquad
  |K| = \sum_{i=1}^N \frac{ H_i |s_i|}{2}.
  \end{align*}
  \begin{figure}[t]
\begin{center}
\includegraphics[scale=0.30]{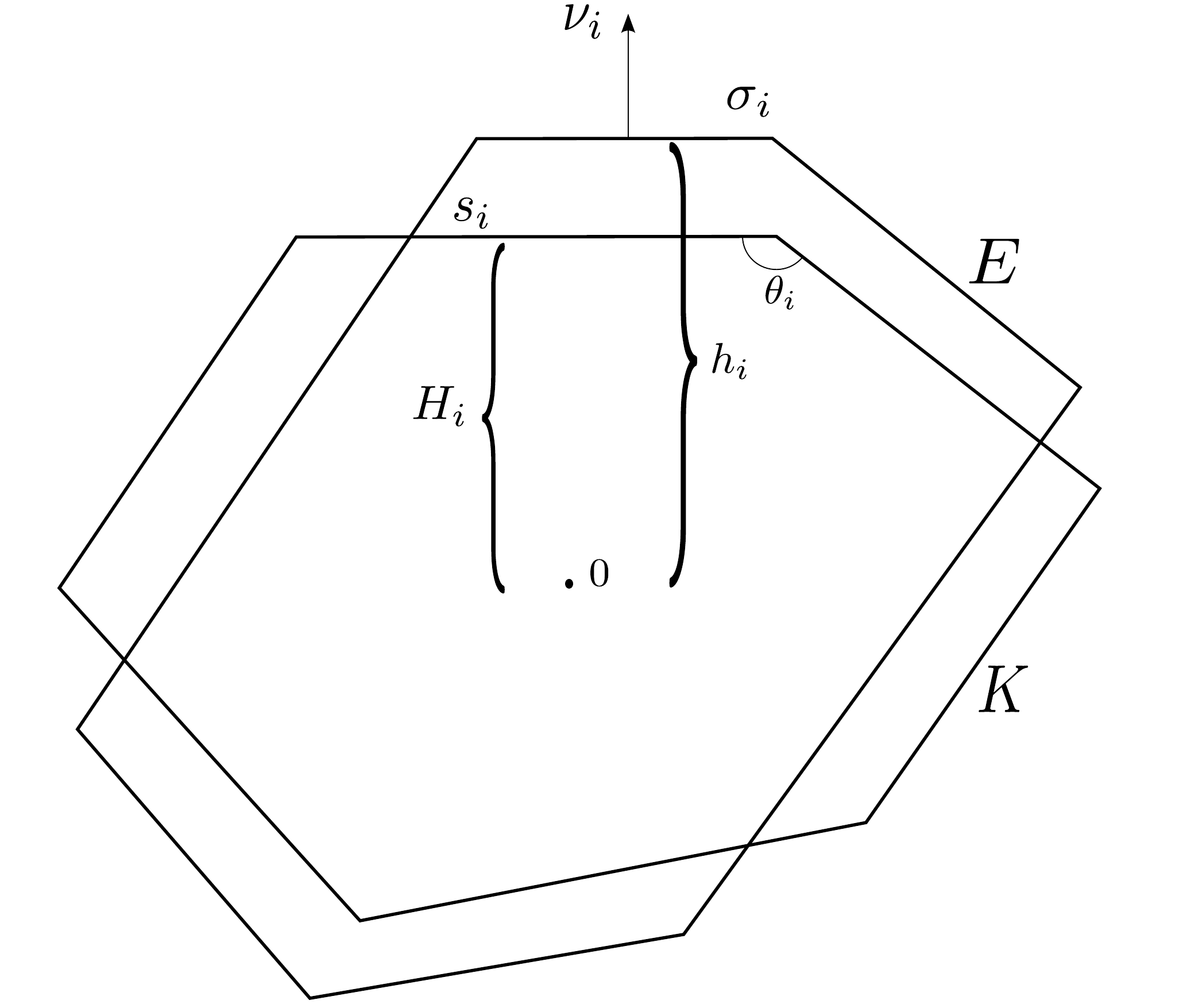}\caption{\small{Notation used for $K$ and a parallel set $E$.}}
\end{center}
\end{figure}
Recalling that
 $\e_i = h_i - H_i$, we may express the volume constraint $|E| = |K|$ as
$$\sum_{i=1}^N \frac{ H_i |s_i|}{2} = |K| = |E| = \sum_{i=1}^N \frac{ H_i |\sigma_i| }{2} + \sum_{i=1}^N \frac{ \e_i |\sigma_i|}{2}.$$
Furthermore,
\begin{equation} \label{deficit} 2|K|\delta_{\PP}(E) =\PP(E) -\PP(K)=\sum_{i=1}^N  H_i (|\sigma_i| - |s_i| )= -\sum_{i=1}^N \e_i |\sigma_i|. \end{equation}
Note that $\sum_{i=1}^N |\e_i| \leq C |E\Delta K|$ for some constant $C=C(f)$, and so by (\ref{FiMPStatement}),
\begin{equation} \label{epsilon}\bigg( \sum_{i=1}^N |\e_i|\bigg)^2 \leq C\delta_{\PP}(E),
\end{equation}
and in particular, $|\e_i|^2 \leq C\delta_{\PP}(E)$ for each $i$.\\

\noindent \textit{Step 1:} We use (\ref{betaform}) and add and subtract $\frac{\PP(K)}{2|K|} = \frac{\g_{\PP}(K)}{2|K|}$ to obtain
\begin{align*}
\BPP(E)^2
\leq \frac{1}{2|K|}\Big(\PP(E)  -\int_E \frac{dx}{f_*(x) }\Big)  &=\delta_{\PP}(E)  + \frac{1}{2|K|} \Big(\int_{K \setminus E} \frac{dx}{f_*(x)}
  - \int_{E \setminus K} \frac{dx}{f_*(x)}\Big).
  \end{align*} 
  Thus we need only to control the term $A-B$ linearly by the deficit, where
\[ A = \int_{K \setminus E}  \frac{dx}{f_*(x) } , \qquad \qquad
B= \int_{E\setminus K} \frac{dx}{f_*(x) }.
\]
To bound the term $A-B$ from above, we bound $A$ from above and bound $B$ from below. Our main tool is the anisotropic coarea formula in the form given in \eqref{awcoarea}.
 
First, we consider the term $A$, where (\ref{awcoarea}) yields
\begin{equation}\label{Coarea A}
A = \int_{K\setminus E} \frac{dx}{f_*(x)} = \int_0^\infty \frac{\PP(rK; K\setminus E)}{r}dr
=\int_0^1 \frac{\PP(rK; K\setminus E)}{r}dr.
\end{equation}
We introduce the notation
$$ I^- = \{ i \in \{1, \hdots N\} : \e_i<0\}, \qquad
I^+ = \{1, \hdots N\} \setminus I^-.$$
From \eqref{Coarea A}, we obtain an upper bound on $A$ by integrating over $r$, for each $i \in I^-$, the part of the perimeter of $rK$ that lies between $\sigma_i$ and $s_i$. This means that for each $r$, we pick up the part of $\partial^*(rK)$ that is parallel to $\sigma_i$ and $s_i$, as well as part of the adjacent sides: 
\begin{equation}\label{seepic}\PP(rK; K\setminus E) \leq  \sum_{I^-}\left[ H_i r |s_i| + H_{i-1}  \frac{(rH_{i} - h_i)}{\sin(\theta_{i-1})} 
 + H_{i+1} \frac{(rH_i - h_i )}{\sin( \theta_i)} \right];
 \end{equation}
see Figure 2 and recall \eqref{Hi}.
  \begin{figure}\label{seeline}
\begin{center}
\includegraphics[scale=0.75]{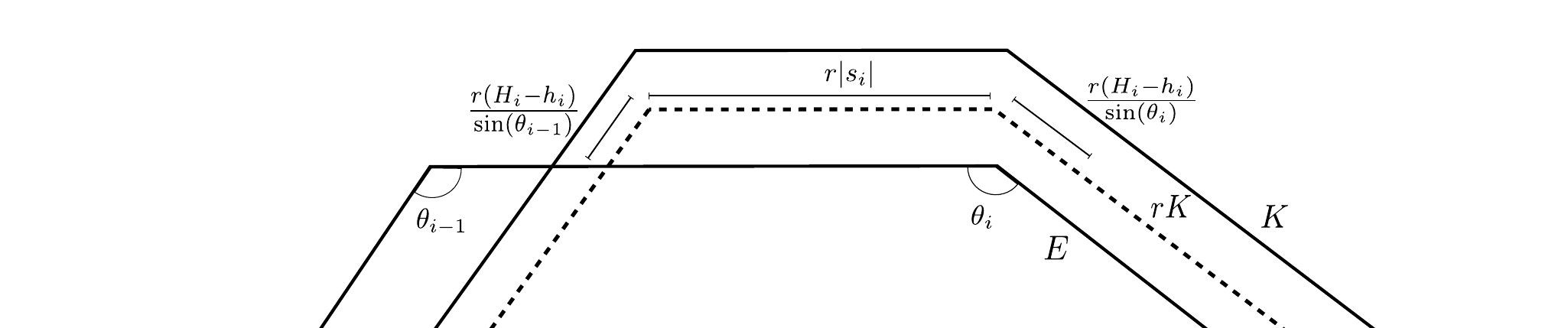}\caption{\small{The surface energy of $rK$ relative to $K\setminus E$ is bounded by the right hand side of \eqref{seepic}.}}
\end{center}
\end{figure}
This and \eqref{Coarea A} imply that 
 \begin{equation}\label{5}
 A
 \leq \sum_{I^-}
  \int_{h_i/H_i}^1
   \left[ H_i r |s_i| + H_{i-1}  \frac{(rH_{i} - h_i)}{\sin(\theta_{i-1})} 
 + H_{i+1} \frac{(rH_i - h_i )}{\sin( \theta_i)} \right] \frac{dr}{r},
 \end{equation}

Now we add and subtract the term  $\int_{h_i/H_i}^1 H_i |\sigma_i| \frac{dr}{r}.$ The idea is that $H_i |\sigma_i|$ gives a rough estimate of the term in brackets on the right hand side of $(\ref{5})$. Indeed, for each $r$, the part of $\partial^*(r K)$ between $\sigma_i$ and $s_i$ has length roughly equal to $H_i|\sigma_i|$. We will see that this estimate is not too rough; the error can be controlled by the deficit. Thus we rewrite \eqref{5} as
\begin{align*}
A& \leq  \sum_{i \in I^-} \int_{h_i/H_i}^1 \frac{H_i |\sigma_i|}{r}\,dr
+\sum_{i \in I^-} \int_{h_i/H_i}^1 H_i |s_i| + \left[ H_i - \frac{h_i}{r} \right] \left( \frac{H_{i-1}}{\sin (\theta_{i-1}) } + \frac{H_{i+1}}{\sin(\theta_i)} \right) - \frac{H_i |\sigma_i|}{r} \,dr.
\end{align*}
Noting that $H_i/\sin(\theta_j) \leq C =C(f)$ for each $i, j$, the right hand side is bounded by $A_1 + A_2$, 
where
\[
A_1 = 
\sum_{i \in I^-} \int_{h_i/H_i}^1 \frac{H_i |\sigma_i|}{r}dr , \qquad A_2 = \sum_{i \in I^- } \int_{h_i/H_i}^1 H_i |s_i| +C \left[ H_i - \frac{h_i}{r} \right]  - \frac{H_i |\sigma_i|}{r}dr.
\]
 The term $A_2$ is the error term that we will show is controlled by the deficit in Step $2$.

 First, we perform an analogous computation for $B$, 
and show how, once the error terms are taken care of,
 the proof is complete.
 Again, by (\ref{awcoarea}), we have
$$B = \int_{E\setminus K} \frac{dx}{f_*(x) }
=
 \int_0^{\infty} \frac{\PP(rK ; E\setminus K)}{r}\, dr=
 \int_1^{\infty} \frac{\PP(rK ; E\setminus K)}{r} \,dr.$$
To bound $B$ from below, we integrate, for each $i\in I^+$, only the part of $\partial^*(rK)$ that is parallel to $s_i$ and $\sigma_i$ and lies between $s_i$ and $\sigma_i$ . We call this segment
 $\ell_i^r := E\setminus K \cap \{e_i + rx_i\}$, where $e_i$ is the vector parallel to the sides $\sigma_i $ and $s_i$, $x_i \in s_i$, and $r \in [1, h_i/H_i]$. 
 
Thus, letting $s_i^r$ be the side of $rK$ parallel to $s_i$ and recalling \eqref{Hi},  we have
$$\int_1^{\infty} \frac{\PP(rK; E\setminus K)}{r} \,dr 
\geq \sum_{i \in I^+} \int_1^{h_i/H_i}\frac{ H_i|s_i^r \cap \ell_i^r| }{r}\,dr.$$
Once again, a rough estimate for $H_i|s_i^r \cap \ell_i^r |$ is given by $H_i |\sigma_i|.$ We will again show that this estimate is not too rough, specifically, that the error between these integrals is controlled by the deficit.
So we continue:
$$B\geq  \sum_{i \in I^+} \int_1^{h_i/H_i} \frac{H_i |\sigma_i| }{r} \,dr
+ \sum_{i \in I^+} \int_1^{h_i/H_i} \frac{H_i|s_i^r \cap \ell_i^r|}{r} - \frac{H_i |\sigma_i|}{r} dr =B_1 +B_2,$$
  where 
\begin{align*}
B_1 = \sum_{i \in I^+} \int_1^{h_i/H_i} \frac{H_i |\sigma_i| }{r} \,dr,\qquad B_2& = \sum_{i \in I^+} \int_1^{h_i/H_i} \frac{H_i|s^r_i \cap \ell_i^r|}{r} - \frac{H_i |\sigma_i|}{r}\, dr.
 \end{align*}
 Like $A_2$, $B_2$ is an error term that we will show is controlled by the deficit in Step $2$. 

Before bounding $|A_2|$ and $|B_2|$ by the deficit, let us see how this will conclude the proof. As we saw,
$\BPP(E)^2 \leq \delta_{\PP}(E) + \frac{1}{2|K|}(A-B). $
Recalling that $h_i = H_i + \e_i$, 
\begin{align*}
A-B &= \sum_{i \in I^-} \int_{h_i/H_i}^1 \frac{H_i |\sigma_i|}{r} dr
    - \sum_{i \in I^+ } \int_1^{h_i/H_i} \frac{H_i |\sigma_i|}{r} dr+ A_2 -B_2\\
& = -\sum_{i \in I^-} H_i |\sigma_i| \log \Big(\frac{h_i}{H_i} \Big)  
- \sum_{i \in I^+} H_i |\sigma_i| \log\Big(\frac{h_i}{H_i} \Big) + A_2- B_2 \\
& =- \sum_{i =1}^N H_i |\sigma_i|\Big(\frac{\e_i}{H_i} + O(\e_i^2)\Big)
 + A_2 - B_2 =  - \sum_{i=1}^{N} \e_i |\sigma_i| + \sum_{i=1}^N O(\e_i^2) +A_2 - B_2 .
\end{align*}
The first term is precisely equal to 
$2|K| \delta_{\PP}(E)$ by (\ref{deficit}), while
 $\sum_{i} O(\e_i^2) \leq C\delta_{\PP}(E)$ by $(\ref{epsilon})$.
 Therefore, once we show that $|A_2|$ and $|B_2|$ are controlled linearly by the deficit, our proof is complete.\\

\noindent \textit{Step 2:} In this step we bound the error terms. We show that  $|A_2| \leq C \delta_{\PP}(E)$; the proof that $|B_2| \leq C \delta_{\PP}(E)$ is analogous.
The main idea for estimating the integral $A_2$
 is to show that the contribution of the adjacent sides is small, and then estimate the rest of integrand slice by slice. Recalling $A_2,$ the triangle inequality gives
\begin{equation} \label{a2} |A_2|\leq \bigg| \sum_{i \in I^-} \int_{h_i/H_i}^1 \frac{H_i}{r} (r|s_i| - |\sigma_i|) dr \bigg|
+  C\sum_{i \in I^-}\bigg| \int_{h_i/H_i}^1 \Big[H_i - \frac{h_i}{r}\Big] dr\bigg|.
\end{equation}
The second term in (\ref{a2}) corresponds to the contribution of adjacent sides. By $h_i = H_i +\e_i$,
\begin{align*}
&C \sum_{ i \in I^-}\bigg| \int_{h_i / H_i}^1 
\left[H_i - \frac{h_i}{r}\right] 
dr \bigg| =  C\sum_{i \in I^-}  \bigg| (H_i - h_i) + h_i \log\Big(\frac{h_i}{H_i}\Big)\bigg|  \\
=& C\sum_{i \in I^-}
 \Big| -\e_i + h_i \frac{\e_i}{H_i} + O(\e_i^2)\Big| =C \sum_{i \in I^-}
 \Big| \frac{\e_i^2 }{H_i} + O(\e_i^2) \Big| =C \sum_{I^-} O(\e_i^2) \leq C \delta_{\PP}(E).
\end{align*}
To bound the first term in \eqref{a2}, we will show that $\big| r|s_i| - |\sigma_i| \big| \leq C
 \max \{ |\e_{i-1}| \}$
for $r \in [h_i/H_i, 1]$, where the constant $C$ depends on $\PP$, and then obtain our bound by integrating.
To this end, we rotate our coordinates such that $\nu_i = e_2$, so the side $s_i$ has endpoints $(a, H_i)$ and $(b, H_i)$ for some $a<b$. We compute explicitly the endpoints of $\sigma_i$; it has, respectively, left and right endpoints
$$\Big( a + \tan\left(\theta_{i-1} - \pi/2 \right) \e_i - \frac{\e_{i-1}}{\sin(\theta_{i-1})} ,\ h_i \Big) \qquad \text{and} \qquad
\Big(b - \tan\left(\theta_i -\pi/2 \right)\e_i + \frac{\e_{i+1}}{\sin(\theta_{i})}, \  h_i \Big).$$
Thus 
$$|\sigma_i| = \Big|b
 - \tan \left(\theta_i -\pi/2 \right) \e_i
  + \frac{\e_{i+1}}{\sin(\theta_{i})} 
  -\Big(a + \tan\left(\sigma_{i-1} - \pi/2\right)
   \e_i 
   - \frac{\e_{i-1}}{\sin(\theta_{i-1})}\Big)\Big| .
   $$
   and so
   $$\|\sigma_i| - |b-a| | \leq C( |\e_i| +|\e_{i+1}| +|\e_{i-1}|),$$
   where $C$ depends on $\PP$.
Therefore, recalling that $|b-a| = |s_i|$,
\begin{align*}
 \Big| r|s_i| - |\sigma_i| \Big| 
& \leq (1-r)|s_i| +C( |\e_i| +|\e_{i+1}| +|\e_{i-1}|) \leq
 \frac{|\e_i|}{H_i} |s_i| + C \max \{ |\e_j |\}\leq C  \max \{ |\e_j |\}.
\end{align*}
Given this estimate on slices, we integrate over $r$:
\begin{align*}
\sum_{i \in I^-}& \int_{h_i / H_i}^1 \frac{H_i}{r} \left(r|s_i| - |\sigma_i|\right)dr \leq  C\max \{ |\e_j |\}  \sum_{i \in I^-} \int_{h_i / H_i}^1 \frac{H_i}{r} dr=C \max \{ |\e_j |\}\sum_{i \in I^-} H_i 
\Big| \log\Big(\frac{h_i}{H_i}\Big)\Big| \\
&= C  \max \{ |\e_j |\}\sum_{i \in I^-}  (\e_i + O(\e_i^2)) = O(\max\{ |\e_j|^2\}) \leq C(\PP) \delta_{\PP}(E),
\end{align*}
where the last inequality follows from \eqref{epsilon}.
\end{proof}
We prove Theorem~\ref{dim2} after introducing the following definition that we will need in the proof.
\begin{definition}\label{special}
A set $E$ is a \textit{volume constrained} $(\e, \eta_0)$-\textit{minimizer of $\PP$} if 
$$\PP(E) \leq \PP(F) + \e |E \Delta F|$$  for all $F$ such that $|E| = |F|$ and $(1-\eta_0)E \subset F\subset (1+\eta_0) E .$
\end{definition}
\begin{proof}[Proof of Theorem~\ref{dim2}]
By Proposition~\ref{poincare}, we need only to show that there exists some $C$ depending on $f$ such that 
\begin{equation}\label{betaonly2}\BPP(E)^2 \leq C\delta_{\PP}(E).
\end{equation}
 for all sets $E$ of finite perimeter with $0<|E|<\infty.$
Suppose for contradiction that  (\ref{betaonly2}) does not hold. There exists a sequence $\{E_j\}$ such that $|E_j| = |K|$, $\delta_{\PP}(E_j) \to 0$, and 
\begin{align}\label{contra2}
\PP(E_j) \leq \PP(K) + c_3 \BPP(E_j)^2
\end{align} 
for $c_3$ to be chosen at the end of this proof. By an argument identical to the one given in the proof of Theorem~\ref{Smooth}, we obtain a new sequence $\{F_j\}$ with $F_j \subset B_{R_0}$ for all $j$ such that the following properties hold:

\begin{itemize}
\item each $F_j $ is a minimizer of $Q_j(E) = \PP(E) + \frac{|K|\m}{8\M}| \BPP(E)^2 - \e_j^2| + \Lambda\big| |E| - |K|\big|$ among all sets $E \subset B_{R_0},$
where $\e_j  = \BPP(E_j)$;
\item
$F_j$ converges in $L^1$ to a translation of $K$;
\item
$|F_j | = |K| \text{ for $j$ sufficiently large}$;
\item the following lower bound holds  for $\BPP(F_j):$
\begin{equation}\label{epsbeta} \e_j^2 \leq 2 \BPP(F_j)^2.\end{equation}
\end{itemize}
Translate each $F_j$ such that $|F_j \Delta K | = {\inf}\{|F_j \Delta (K+y)|: y\in\mathbb{R}^2\}$.
We claim that  for all $\e >0$, there exists $\eta_0>0$ such that $F_j$ is a volume constrained $(\e, \eta_0)$-minimizer of $\PP$ (Definition~\ref{special}) for $j$ large enough. 
Indeed, fix $\e>0$ and let $\eta_1 = c_1 \e$, where $c_1 = c_1(f)$ will be chosen later. 
 By Lemma~\ref{3.1}, there exists $\eta_2 $ such that if $(1- \eta_2)K \subset E \subset (1+  \eta_2)K,$ 
 then $|y_{E}| <\eta_1$. Let $\eta_0 =\min \{ \eta_1, \eta_2\}/2$. 

By Lemma~\ref{UDE}, each $F_j$ satisfies uniform density estimates, and so Lemma~\ref{bound} implies that,  for $j$ large, $(1- \eta_0)K \subset F_j \subset (1+ \eta_0)K$ and thus $|y_{F_j}|<\eta_1$.
Let $E$ be such that $|E| = |F_j|$ and $(1-\eta_0)F_j \subset E \subset (1+\eta_0)F_j$. Then $|y_E| <\eta_1$ and
\begin{align*}
(1- \eta_1)K \subset F_j \subset (1+\eta_1)K, \qquad (1- \eta_1)K \subset &E \subset (1+  \eta_1)K.
\end{align*}
Because $F_j$ minimizes $Q_j$, 
$$\PP(F_j) + \frac{|K|\m}{4\M} |\BPP(F_j)^2 -\e_j^2| \leq \PP(E) + \frac{|K|\m}{4\M} |\BPP(E)^2 - \e_j^2|$$
and so by the triangle inequality and since $\m \leq \M$, 
$$\PP(F_j) \leq \PP(E) + \frac{|K|}{4} |\BPP(E)^2 - \BPP(F_j)^2| .$$
If $\PP(F_j) \leq \PP(E)$, then the volume constrained minimality condition holds trivially. Otherwise, with a bound as in \eqref{bbound}, we have
$$\PP(F_j) \leq \PP(E) + \frac{\PP(F_j) - \PP(E)}{2} + \frac{|\g_{\PP}(E) - \g_{\PP}(F_j)|}{2}  .$$
and so
$$\PP(F_j) \leq \PP(E) + |\g_{\PP}(E) - \g_{\PP}(F_j)| .$$

As in the proof of Theorem~\ref{Smooth}, the H\"{o}lder modulus of continuity for $\g_{\PP}$ shown in Proposition~\ref{convergence}(2) does not provide a sharp enough bound on the term $ |\g_{\PP}(E) - \g_{\PP}(F_j)| $; we must show that $\g_{\PP}$ is Lipschitz when the centers of $E$ and $F_j$ are bounded away from their symmetric difference. In this case, we must be more careful and show that the Lipschitz constant is small when $|E| = |F|$ and $E$ and $F_j$ are $L^{\infty}$ close.
If $\g_{\PP}(E) \geq \g_{\PP}(F_j)$, then using \eqref{gamma}, we have 
\begin{align*} 
\g_{\PP} (E ) - \g_{\PP}(F_j) & \leq \int_E \frac{dx}{f_*(x-y_E)} - \int_{F_j} \frac{dx}{f_*(x-y_E)}= \int_{E\setminus F_j} \frac{dx}{f_*(x-y_E)} - \int_{F_j\setminus E} \frac{dx}{f_*(x-y_E)}.
\end{align*}
One easily shows from the definition that for any $x, y \in \Rn$, 
$$f_*(x) - \frac{1}{m_{\PP}} |y| \leq f_*(x-y) \leq f_*(x) + \frac{1}{\m}|y|.$$
Therefore, since
$(1- \eta_1) K \subset E\Delta F_j \subset (1+ \eta_1) K$ and $|y_E| \leq \eta_1,$
$$
1- \eta_1 (1 +1/\m) \leq f_*(x-y_E) \leq 1+ \eta_1 (1 +1/\m)
$$
 for $x \in E\Delta F_j$, implying that
\begin{align*} 
\g_{\PP}(E) - \g_{\PP}(F_j)
&  \leq \int_{E\setminus F_j} \frac{dx}{1-\eta_1 (1+ 1/\m)} 
- \int_{F_j\setminus E}\frac{dx}{1+\eta_1(1+1/\m )} 
 \leq C \eta_1 |E\Delta F_j|.
\end{align*}
where $C= 1+ 1/\m$.
The analogous argument holds if $\g_{\PP}(E) \leq \g_{\PP}(F_j),$
and so
$$\PP(F_j) \leq \PP(E) +C \eta_1 |E\Delta F_j|.$$
Letting $c_1 =1/C$, we conclude that $F_j$ is a volume constrained $(\e, \eta_0)$-minimizer of surface energy, and 
for $j$ large enough, $(1- \eta_0/2)K \subset F_j \subset (1+ \eta_0/2)K$ by Lemma~\ref{bound}.
Therefore, Theorem~\ref{figmag} below implies that,  for $j$ sufficiently large, $F_j$ is a convex polygon with 
$\nu_{F_j}(x) \in \{v_i\}_{i=1}^N$ for $\mathcal{H}^{1}\text{-a.e. } x \in \partial F_j.$  Moreover, for any $\eta,$ $(1- \eta)K \subset F_j \subset (1+ \eta)K$ for $j$ large enough, so actually $\{v_{F_j}\} = \{v_i\}_{i=1}^N$ for $j$ sufficiently large. In other words, for $j$ large enough, $F_j$ is parallel to $K$, so Proposition~\ref{crystalprop} implies that
\begin{equation}
\BPP(F_j)^2 \leq C_1 \delta_{\PP}(F_j),
\end{equation}
where $C_1$ depends on $f$.
On the other hand, $F_j$ minimizes $Q_j$, so 
comparing against $E_j$ and using (\ref{contra2}) and (\ref{epsbeta}) implies
\begin{align*}
\PP(F_j) & \leq \PP(E_j) 
\leq \PP(K)  + c_3 \e_j^2 \leq  \PP(K) +2c_3\BPP(F_j)^{2}.
 \end{align*}
By \eqref{epsbeta}, $\BPP(F_j)>0$, so
choosing $c_3$ small enough such that $c_3 < |K|/C_1$, we reach a contradiction.
\end{proof}

\begin{theorem}\cite[Theorem 7]{figallimaggi11}\label{figmag}
Let $n=2$ and let $f$ be a crystalline surface tension. There exists a constant $\e_0$
 such that if, for some $\eta>0$ and some $0<\e<\e_0$ , $(1-\eta/2)K \subset E \subset (1+\eta/2)K$ and $E$ is a volume constrained $(\e, \eta)$-minimizer, then $E$ is a convex polygon with 
$$\nu_E(x) \in \{v_i\}_{i=1}^N \qquad \text{for }\mathcal{H}^{1}\text{-a.e. } x \in \partial E.$$
\end{theorem}
\begin{remark}\rm{In \cite[Theorem 7]{figallimaggi11}, Figalli and Maggi assume that $E$ is a volume constrained $(\e, 3)$-minimizer (and actually, their notion of $(\e, 3)$-minimality is slightly stronger than ours). However, by adding the additional assumption that $(1-\eta/2) K \subset E \subset (1+\eta/2)K$, it suffices to take $E$ to be a volume constrained-$(\e, \eta)$ minimizer (with the definition given here) with $\eta$ as small as needed. Indeed, if $(1-\eta/2) K \subset E \subset (1+\eta/2)K$, then $(1-\eta)E \subset$ co$(E) \subset  (1+\eta)E$ where  co$(E)$ is the convex hull of $E$. Then, in the proof of \cite[Theorem 7]{figallimaggi11}, the only sets $F$ used as comparison sets are such that $|E| = |F|$ and $(1-\eta)E\subset F\subset (1+\eta)E$. }
\end{remark}

\section{Another form of the oscillation index}\label{other}
The oscillation index $\BPP(E)$ is the natural way to quantify the oscillation
 of the boundary of a set $E$ relative to the Wulff shape $K$ for a given surface energy $\PP$,
  as it admits the stability inequality \eqref{statement1} 
  with a power that is independent of $f$. 
One may wonder if it would be suitable to quantify the oscillation of $E$ by looking at the Euclidean distance between normal vectors of $E$ and corresponding normal vectors of $K$. 
While such a quantity may be useful in some settings, in this section we show that it does not admit a stability result with a power independent of $f$. This section examines the term $\BPP^*(E)$ defined in \eqref{betastar} and gives two examples showing a failure of stability. We then give a relation between $\BPP $ and $\BPP^*$ for $\g$-$\lambda$ convex surface tensions. As a consequence of Theorem~\ref{th1}, this implies a stability result for $\BPP^*$, though, as the examples show, there is a necessary dependence on the $\g$-$\lambda$ convexity of $f$.

The following example illustrates that there does not exist a power $\sigma$ such that 
\begin{equation}\label{noway}
\beta^*_{\PP}(E)^{\sigma} \leq C(n,f) \delta_{\PP}(E)
\end{equation}
for all sets $E$ of finite perimeter with $0<|E|<\infty$ and for all surface energies $\PP$.
\begin{example}\label{crystalexample} \rm{
In dimension $n=2$, we construct a sequence of Wulff shapes $K_{\theta}$
 (equivalently, a sequence of surface energies $\PP_{\theta}$) and a sequence of sets $E_{\theta}$ such that $\delta_{{\theta}} (E_{\theta}) \to0$ but $\beta^*_{{\theta}}(E_{\theta}) \to \infty$  as $\theta \to 0$. We use the notation $ \delta_{\theta}=\delta_{\PP_{\theta}}$ and $ \beta^*_{\theta}=\beta^*_{\PP_{\theta}}$.

We let $K_{\theta}$ be a unit area rhombus where one pair of opposing vertices has angle $\theta<\frac{\pi}{4} $ and the other has angle $\frac{\pi}{2} - \theta$. The length of each side of $K_{\theta}$ is proportional to $\theta^{-1/2}$. 
 Let $L = \theta^{-1/4}$.
We then construct the sets $E_{\theta}$ by cutting away a triangle with a zigzag base and with height $L$ from both corners of $K_{\theta}$ with vertex of angle $\theta$ (see Figure $3$). We choose the zigzag so that each edge in the zigzag is parallel to one of the adjacent edges of $K_{\theta}$. By taking each segment in the zigzag to be as small as we wish, we may make the area of each of the two zigzag triangles arbitrarily close to the area of the triangle with a straight base, which is 
$$A = L^2\tan(\theta/2) = \theta^{-1/2}\tan(\theta/2)\approx \theta^{1/2},$$
as this triangle has base $2L \tan(\theta/2).$ Both of the other two sides of the triangle have length $m= L/\cos(\theta/2).$
\begin{figure}[t]
\begin{center}
\includegraphics[scale=0.44]{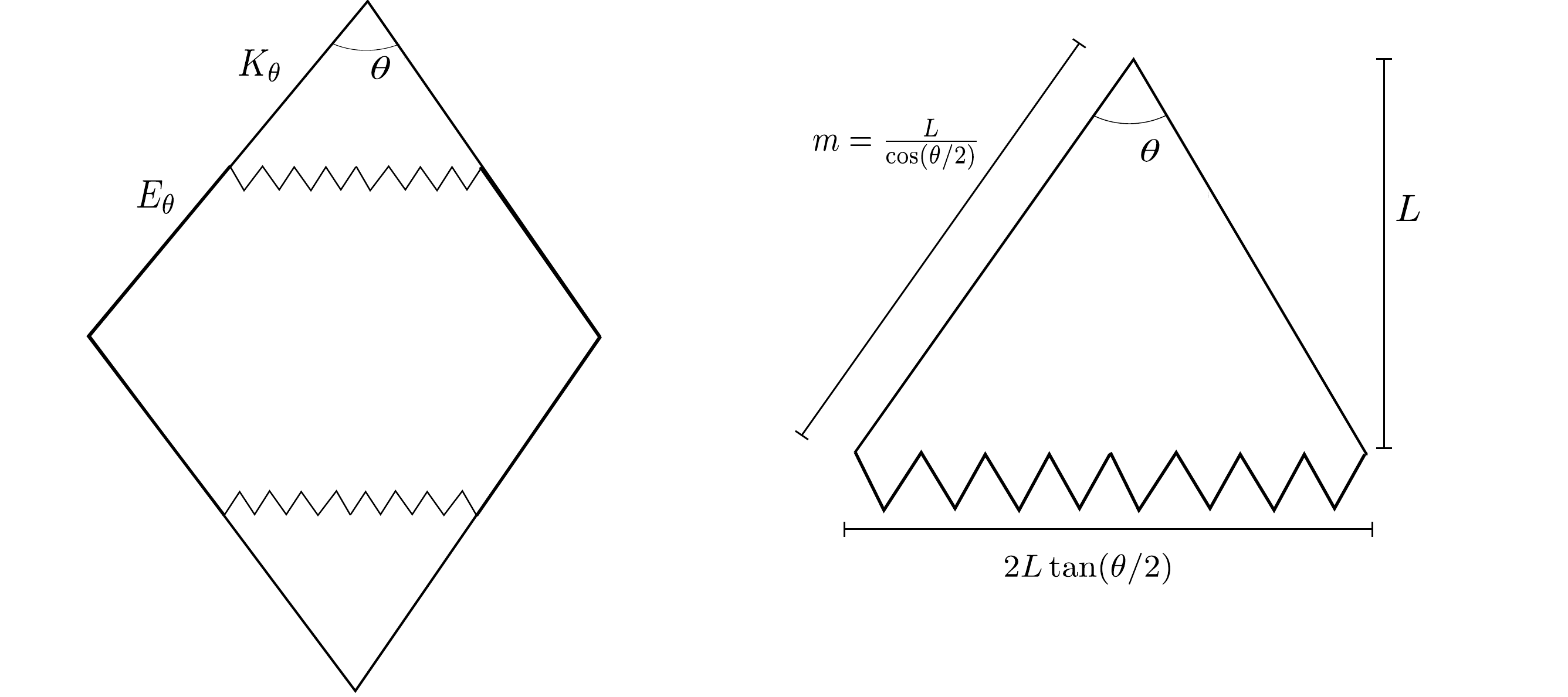}\caption{\small{The sets $E_{\theta}$ are formed by cutting away a zigzag triangle from the top and bottom of $K_{\theta}$  have $\delta_{\theta}(E_{\theta}) \to 0$ but $\beta^*_{{\theta}}(E_{\theta}) \to \infty$  as $\theta \to 0$.}}
\end{center}
\end{figure} 

Let us now compute the deficit $\delta_{\theta}$ and the Euclidean oscillation index $\beta_{{\theta}}^*$ of $E_{\theta}$. By construction, $\PP_{\theta}(E_{\theta}) = \PP_{\theta}(K_{\theta}) = 2$, and therefore
$$\delta_{{\theta}}(E_{\theta}) = \frac{2}{2(1- A)^{1/2}}  - 1  = \frac{1}{(1- A)^{1/2}} -1 = \theta^{1/2} + o(\theta^{1/2}).$$
To compute $\beta^*_{\theta} (E_{\theta})^2$, we cannot characterize the point $y$ for which the minimum in \eqref{betastar} is attained in general. However, something may be said for an $n$-symmetric set, i.e., a set $E$ that for which there exist $n$ orthogonal hyperplanes such that $E$ is invariant under reflection with respect to each of them. The intersection of these orthogonal hyperplanes is called the \textit{center of symmetry of $E$}. Indeed, a slight variation in the proof of \cite[Lemma $5.2$]{maggi2008some} shows that
\begin{equation}\label{3beta}3\BPP^*(E) \geq \bigg( \frac{1}{n|K|^{1/n}|E|^{1/n'} }\int_{\partial^* E} 1-\nu_E(x) \cdot \nu_K\Big(\frac{x-z}{f_*(x-z)}\Big)\, d\Hn(x) \bigg)^{1/2}.
\end{equation}
 where $z$ is the center of symmetry of $E$. 
By construction, $E_{\theta}$ is a $2$-symmetric set with center of symmetry $0$, so
\begin{align*}9\beta^*_{{\theta}} (E_{\theta})^2 
&\geq
 \frac{1}{2(1-A)^{1/2}} \int_{Z} 1 -\nu_{E_{\theta}} (x) \cdot \nu_{K_{\theta}}\Big(\frac{x}{f_*(x)}\Big)\, d \mathcal{H}^{1} \geq 
  \frac{1}{2} \int_{Z} 1 -\nu_{E_{\theta}} (x) \cdot \nu_{K_{\theta}}\Big(\frac{x}{f_*(x)}\Big)\, d \mathcal{H}^{1},
 \end{align*}
where $Z$ denotes the union of the two zigzags. By construction, $\mathcal{H}^1(Z)$ is exactly equal to $\mathcal{H}^{1}(\partial K_{\theta} \setminus \partial E_{\theta}) = 4m$. Moreover, because the edges of $E_{\theta}$ are parallel to those of $K_{\theta}$, we find that 
$$1 -\nu_{E_{\theta}} (x) \cdot \nu_{K_{\theta}}\Big(\frac{x}{f_*(x)}\Big) = \begin{cases} 0 & x \in Z_1\\
1-\cos(\pi - \theta) & x \in Z_2
\end{cases}
$$
where $Z_1$ is the set of $x\in Z$ where $\nu_{E_{\theta}}(x)$ is equal to 
$\nu_{K_{\theta}}(\frac{x}{f_*(x)})$ and $Z_2$ is the set of $x \in Z$ 
where $\nu_{E_{\theta}}(x)$ is equal to the normal vector to the other side of $K_{\theta}.$ Moreover, we have constructed $E_{\theta}$ so that $\mathcal{H}^1(Z_1) = \mathcal{H}^1(Z_2) = 2 m.$ Thus, as $\theta<\frac{\pi}{4},$
  \begin{align*} 
\beta^*_{{\theta}} (E_{\theta})^2 & \geq   \frac{1}{2}  \int_{Z_2} 1 - \cos (\pi - \theta) \ d \mathcal{H}^{1}  \geq \frac{\mathcal{H}^1(Z_2)}{2} = m =1/(\theta^{1/4}\cos(\theta/2))\to \infty
\end{align*}
as $\theta \to 0$.
Therefore, for any exponent $\sigma$, the inequality (\ref{noway}) fails to hold; we may choose $\theta$ sufficiently small such that $E_{\theta}$ is a counterexample.}
\end{example}
 
The next example shows that even if we restrict our attention to surface energies that are $\g$-$\lambda$ convex (Definition~\ref{gammalambda}), an inequality of the form in (\ref{noway}) cannot hold with an exponent smaller than $\sigma = 4.$
The example is presented in dimension $n=2$ for convenience, though the analogous example in higher dimension also holds.
\begin{example}\label{pexample} \rm{ Fix $p>2$ 
and define the  surface tension $f_p(x) = \left( |x_1|^p+ |x_2|^p \right)^{1/p}$ to be the $\ell^p$ norm in $\mathbb{R}^2$. We show below that $f_p$ is a $\g$-$\lambda$ convex surface tension. H\"{o}lder's inequality ensures that the support function $f_*$ is given by $f_q$, in the notation above, where $q$ is the H\"{o}lder conjugate of $p$. The Wulff shape $K= \{f_q(x) <1\}$ is therefore the $\ell^q$ unit ball. We let $\PP_p$ denote the surface energy corresponding to the surface tension $f_p$.

We build a sequence of sets $\{E_{r}\}$ depending on $p$ such that, for any $\sigma<4$, we may choose $p$ large enough so that $\delta_{p}(E_{r})/\beta^*_{p}(E_{r})^{\sigma}\to 0$ as $r \to 0$.
Here we use the notation $\beta^*_{p}(E) = \beta^*_{\PP_p}(E) $ and $\delta_p(E) = \delta_{\PP_p}(E).$
We may locally parameterize $K$ near $(0,1)$ as the subgraph of the function 
$v_q(x_1) = \left(1 - |x_1|^q\right)^{1/q}.$ 
Thus  
$v_q'(x_1) =  -|x_1|^{q-2}x_1/{(1- |x_1|^q )^{1/p}}$ and
\begin{equation}\label{nuKK}
\nu_K((x_1, \ v_q(x_1))) = \frac{\Big(\frac{|x_1|^{q-2}x_1}{(1- |x_1|^q )^{1/p}}, \ 1\Big) }{\sqrt{1+
\frac{|x_1|^{2q-2}}{(1-|x_1|^q )^{2/p}} }} =\frac{(|x_1|^{q-2}x_1 + O(|x_1|^{2q-1}) , \ 1) }{\sqrt{1+|x_1|^{2q-2} + O(|x_1|^{3q-2} ) }}
\end{equation}
The sets $E_r$ are formed by replacing the top and bottom of $K$ with cones. More precisely, let $\textbf{C}_r = (-r,r) \times \mathbb{R}$. We form $E_r$ by replacing $\partial K \cap \textbf{C}_r$ with the graphs of the functions $w$ and $-w$, where $w_1: (-r, r) \to \mathbb{R}$ is defined by $w (x_1)  =  -r^{q-1} |x_1|/{(1-r^q)^{1/p}} + C_0,$ with $C_0= (1-r^q)^{1/q } + r^q/{(1-r^q)^{1/p}}.$  The constant $C_0$ is chosen so that $w(r) = v_q(r)$ and $w(-r) = v_q(-r)$.
\begin{figure}[t]
\begin{center}
\includegraphics[scale=0.45]{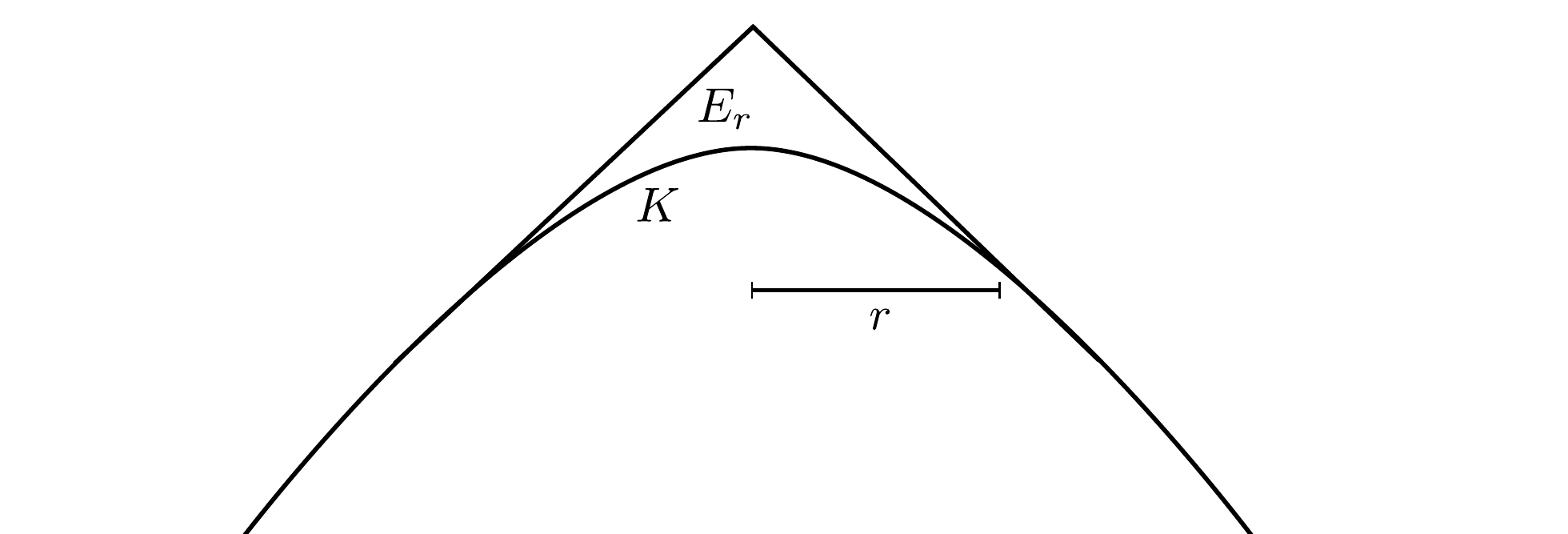}\caption{\small{The sets $E_{r}$ formed by replacing the top and bottom of the $\ell^q$ unit ball with a cone show that \eqref{noway} cannot hold for $\sigma<4$.}}
\end{center}
\end{figure}
For $x_1 \in (-r, r)$ for $r<1$, we have
$w'(x_1) = - r^{q-1}\text{sgn}(x_1)/(1-r^q)^{1/p}$ and 
\begin{equation}\label{nuCC}
\nu_E\left((x_1, \ w(x_1))\right) = \frac{ \Big( \text{sgn}(x_1)  \frac{ r^{q-1}}{(1-r^q)^{1/p}} ,\ 1  \Big) }{\sqrt{ 1 + 
\frac{ r^{2q-2}}{(1-r^q)^{2/p}}}} 
=\frac{( \text{sgn}(x_1)  r^{q-1} + O(r^{2q-1}), \ 1)  }{\sqrt{ 1 + r^{2q-2} +O(r^{3q-2})}}.
\end{equation}
Now, $\PP_p(E_r ) = \PP_p(K) + \PP_p(E_r; \textbf{C}_r) - \PP_p(K; \textbf{C}_r)$, so 
 \begin{align*}
\PP_p(E_r) - \PP_p(K)&
= \int_{-r}^r \Big( \frac{r^q}{1-r^q} +1\Big)^{1/p } - \Big( \frac{|x_1|^q}{1-|x_1|^q} +1\Big)^{1/p } dx_1\\
&=\frac{1}{p} \int_{-r}^r r^q - |x_1|^q + O(r^{2q})  \ dx_1 =Cr^{q+1} + o(r^{q+1} ).
 \end{align*}
The graph of $w$ lies above the graph of $v_q$ for all $|x_1|<r$, so $|E_r| >|K|.$  This implies that $$\delta_{p}(E_r) \leq \frac{\PP_p(E_r) - \PP_p(K)}{2|K|} =Cr^{q+1} + o(r^{q+1} ).$$

Next we compute $\beta_{p}^*(E_r)$ in several steps.
 As in Example~\ref{crystalexample}, $E_r$ is a $2$-symmetric set with center of symmetry $0$, thus it is enough to compute the right hand side of \eqref{3beta}. 
 First, the Taylor expansions in \eqref{nuKK} and \eqref{nuCC} imply that, for $x\in {\bf{C}}_r \cap \partial^* E$, $ \nu_E(x) \cdot \nu_K\big( \frac{x}{f_*(x)} \big)$ is given by
\begin{align*}
&\frac{(|x_1|^{q-2}x_1 + O(|x_1|^{2q-1}) , \ 1) }{\sqrt{1+|x_1|^{2q-2}+ O(|x_1|^{3q-2}) }}\cdot \frac{( \text{sgn}(x_1)  r^{q-1} + O(r^{2q-1}), \ 1)}{\sqrt{ 1 + 
 r^{2q-2} + O(r^{3q-2)}}} = \frac{ 1 + |x_1|^{q-1}r^{q-1} + O(r^{3q-2} )}{\sqrt{(1+|x_1|^{2q-2}+ r^{2q-2} + O( r^{4q-4}))}}\\
 &=  1 + |x_1|^{q-1}r^{q-1} -\frac{1}{2} (|x_1|^{2q-2}+ r^{2q-2} ) + O( r^{3q-2})=1 -\frac{1}{2}( |x_1|^{q-1}-r^{q-1})^2+O( r^{3q-2} ).
\end{align*}
For $x \in \partial^* E \setminus {\bf{C}}_r$,  $ \nu_E(x) \cdot \nu_K\big( \frac{x}{f_*(x)} \big) = 0$.
Hence, 
\begin{align*}
\bigg( \frac{1}{2} \int_{\partial^*E}& \left|\nu_E(x) - \nu_K\Big(\frac{x}{f_*(x)}\Big)\right|^2\, d\mathcal{H}^{1}\bigg)^{1/2} =\bigg( \int_{\partial^*E \cap {\bf{C}}_r} 1 - \nu_E \cdot \nu_K\big(\frac{x}{f_*(x)} \big)\,d\mathcal{H}^{1}\bigg)^{1/2}\\
&  =\bigg( \int_{-r}^r \frac{1}{2} ( |x_1|^{q-1}-r^{q-1})^2 \sqrt{ 1 + r^{2q-2} +O(r^{3q-1})} +  O(r^{3q-2} ) 
\, dx_1\bigg)^{1/2} \\
&=\bigg( \int_{-r}^r \frac{1}{2} ( |x_1|^{q-1}-r^{q-1})^2  +  O(r^{3q-2} ) 
\, dx_1\bigg)^{1/2} 
=Cr^{q-1/2} + o(r^{q-1/2}).
\end{align*}
Furthermore, $|E| = |K| + o(1)$,
so $\sqrt{2}|K|^{-1/4} |E|^{-1/4}= \sqrt{2}|K|^{-1/2} +o(1)$, and so
\begin{align*}\beta_p^*(E_r)&
 = \frac{1}{\sqrt{2}|K|^{1/4} |E|^{1/4}}
  \bigg( \frac{1}{2} \int_{\partial^* E} \Big|\nu_E (x) - \nu_K\Big(\frac{x}{f_*(x)} \Big)\Big|^2 d \mathcal{H}^1\bigg)^{1/2} =C r^{q -1/2} + o(r^{q -1/2}).
\end{align*}
Therefore,
$$
\frac{\delta_{p}(E_r) }{\beta_p^*(E_r)^{\sigma} } \approx \frac{r^{(q+1)}}{r^{\sigma(q - 1/2)}} = r^{ q+1 - \sigma q +\sigma/2}.$$ 
This quantity goes to $0$ as $r$ goes to zero if and only if 
$q+1 -\sigma q+ \sigma/2 >0,$
or, equivalently, if and only if $\frac{2 +\sigma}{2(\sigma-1)}>q$.
For any $\sigma<4$ we may find $1<q <\frac{2 + \sigma}{2(\sigma-1)}.$ Therefore, for any $\sigma<4,$ there exists a $\gamma$-$\lambda$ convex surface tension $f$ such that a bound of the form $\delta_{\PP}(E) \geq C\BPP^*(E)^{\sigma}$ fails.}
\end{example}

When $f$ is $\g$-$\lambda$ convex (recall Definition~\ref{gammalambda}), we can control $\beta^*_{\PP}(E)$ by $\BPP(E)$. As one expects after the previous example, the exponent in this bound depends on the $\g$-$\lambda$ convexity of $\PP$. Indeed, this is the content of Theorem~\ref{th2}. First, we show that the $\ell_p$ norms $f_p$ as defined in the previous example are $\g$-$\lambda$ convex for each $p\in(1, \infty)$.
In the case where $1<p\leq 2$, $f_p$ is actually \textit{uniformly} convex in tangential directions, so it is $\g$-$\lambda$ convex with $\g =0.$ 
Indeed, $f_p(\nu+\tau) = f_p(\nu) + \nabla f_p(\nu)\tau + \frac{1}{2}\int_0^1 \nabla^2 f_p(\nu + s\tau)[\tau , \tau] ds,$ and thus
$$f_p(\nu+\tau) +f_p(\nu-\tau) - 2 f_p(\nu)  = \frac{1}{2} \int_{-1}^{1} \nabla^2f_p(\nu+ s\tau) [\tau, \tau] ds.$$
We can bound the integrand from below pointwise. 
We compute
$$
\partial_{ii} f_p(\nu ) = (p-1) \Big( \frac{ |\nu_i|^{p-2}}{f_p(\nu)^{p-1}} - \frac{ |\nu_i|^{2p-2}}{f_p(\nu)^{2p-1}} \Big) , \qquad
\partial_{ij}f_p(\nu)  =(1-p) \frac{| \nu_i|^{p-2}\nu_i |\nu_j|^{p-2} \nu_j}{ f_p(\nu)^{2p-1}}.$$
Therefore, if $f_p(\nu) = 1,$ then 
$$\nabla^2 f_p(\nu) = (p-1) \sum_{i=1}^{n} |\nu_i|^{p-2} e_i \otimes e_i - (p-1) \sum_{i,j=1}^n  
| \nu_i|^{p-2}\nu_i |\nu_j|^{p-2} \nu_j e_i \otimes e_j$$
and so
$$\nabla^2f_p(\nu)[\tau, \tau] =
 (p-1)\sum_{i=1}^{n} |\nu_i|^{p-2} \tau_i^2 - (p-1) \Big(\sum_{i=1}^n  
| \nu_i|^{p-2}\nu_i \tau_i \Big)^2 .$$
It is enough to consider $\tau$ such that $\tau$ is tangent to $K_p=\{f_p<1\}$ at $\nu$, as $f_p$ is positive $1$-homogeneous and the span of $\nu$ and $T_{\nu}K_p$ is all of $\Rn$.
Observe that $\nabla f_p(\nu) = \sum_{i=1}^n |\nu_i|^{p-2} \nu_i e_i$;
 this is verified by the fact that the support function of $f_p$ is $f_q$, and that
  $\nabla f_p(\nu) = \frac{x}{f_q(x)} $ such that $\frac{x}{f_q(x)} \cdot \nu = f_p(\nu) = 1.$ Thus $\tau$ is tangent to $K_p$ at $\nu$  if and only if
$\tau \cdot \nabla f_p(\nu) = \sum_{i=1}^n  
| \nu_i|^{p-2}\nu_i \tau_i = 0.$ Therefore, for such $\tau$, 
$$\nabla^2f_p(\nu)[\tau, \tau] =
 (p-1)  \sum_{i=1}^{n} |\nu_i|^{p-2} \tau_i^2 \geq (p-1) |\tau|^2.$$
In the case where $p\geq 2$, we use Clarkson's inequality, which states that for $p \geq 2$,
$$f_p\Big(\frac{x+y}{2}\Big)^p +f_p\Big(\frac{x-y}{2}\Big)^p \leq \frac{f_p(x)^p}{2}  +\frac{ f_p(y)^p}{2}.$$
For $\nu$ such that $f_p(\nu) =1$
 and $\tau $ tangent to $K_p $ at $\nu$ with $f_p(\tau)=1,$
Clarkson's inequality with $x = \nu + \e \tau$ and $y = \nu - \e \tau$ implies
$$2\e^p \leq f_p(\nu + \e \tau)^p + f_p(\nu - \e \tau) - 2.$$
This is almost the condition we need, except we  have $f_p^p$ instead of $f_p$ for the terms on the right hand side. 
Note that both $f_p(\nu + \e \tau)$ and $f_p(\nu - \e \tau)$ are greater than $1$, as moving in the tangent direction to $K_p= \{f_p <1\} $  increases $f_p$.
The function $z^p$ is convex with derivative $pz^{p-1}$, so $z^p \leq 2^{p-1}pz + (2^{p-1}p-1)$ for all $z\in [1,2]$. Applying this to $z_1 = f_p(\nu+\e \tau)$ and $z_2 = f_p(\nu-\e \tau)$ yields
$$2\e^p \leq  2^{p-1}p f_p(\nu+\e \tau)^p +2^{p-1}pf_p(\nu - \e \tau)^p -2(2^{p-1}p).$$
Thus $f_p$ is $\g$-$\lambda$ convex with $\g = p-2$ and $\lambda =1/(2^{p-2}p).$

The following lemma about $\g$-$\lambda$ convexity condition will be used in the proof of Theorem~\ref{th2}.

\begin{lemma}\label{lem5} Assume that $f$ is $\g$-$\lambda$ convex.
 Then for all $\nu, \tau \in \Rn$ such that $\nu \neq 0$, 
\begin{equation}\label{frombelow}f(\nu + \tau ) \geq \frac{\lambda}{2^{2+\g}|\nu|} \Big|\tau -\Big(\tau \cdot\frac{\nu}{|\nu|}\Big) \frac{\nu}{|\nu|}\Big|^{2+\g}
+f(\nu) +\nabla f(\nu)\cdot \tau,
\end{equation}
\end{lemma}
\begin{proof} Note that if $f$ is $\g$-$\lambda$ convex, then $f$ is convex. 
To see that \eqref{frombelow} holds for given $\nu_0$ and $\tau_0$, we let $\tilde{f}(\nu) = f(\nu) - f(\nu_0) - \nabla f(\nu_0)\cdot(\nu - \nu_0).$
At the midpoint $\nu_0 + \frac{\tau_0}{2}$, the $\g$-$\lambda$ convexity condition gives us the following:
$$
\tilde{f}(\nu_0) + \tilde{f}(\nu_0 + \tau_0) - 2\tilde{f}(\nu_0 +\frac{\tau_0}{2} ) \geq \frac{\lambda}{|\nu_0|} \Big|\frac{\tau_0}{2} - \Big(\frac{\tau_0}{2}\cdot \frac{\nu_0}{|\nu_0|}\Big) \frac{\nu_0}{|\nu_0|}\Big|^{2+\g}.$$
Convexity implies that $\tilde{f}(\nu_0 + \frac{\tau_0}{2}) \geq 0$, and $\tilde{f}(\nu_0) = 0$ by definition of $\tilde{f}$, implying \eqref{frombelow}.
\end{proof}
Finally, we prove Theorem~\ref{th2}.
\begin{proof}[Proof of Theorem~\ref{th2}]
The quantity $\beta^*_{\PP}(E)$ 
measures the overall size of the Cauchy-Schwarz deficit on the boundary of $E$, while $\BPP(E)$ measures the overall deficit in the Fenchel inequality.
 Our aim is to obtain a pointwise bound of the Cauchy-Schwarz deficit functional by the Fenchel deficit functional, and then integrate over the reduced boundary of $E$. Without loss of generality, we may assume that $|E|= |K| = 1 $ and $E$ has center zero in the sense defined in Section~\ref{abg}.

We fix $x\in \partial^*E$ and consider the Fenchel deficit functional 
$G(\nu) = f(\nu) -\nu \cdot \frac{x}{f_*(x)}$, which possesses the properties that $G(\nu)\geq 0$ and $G(\nu )= 0$ if and only if $\nu =c\, \nabla f_*(x)$ for some $c>0$.

Let $w=\frac{\nabla f_*(x)}{|\nabla f_*(x)|} = \nu_K(\frac{x}{f_*(x)}).$ 
Lemma~\ref{lem5}, with $\nu = w $ and $\tau = \nu_E -w$, implies that
$$f(\nu_E )
 \geq \frac{\lambda}{2^{2+\g}} |(\nu_E - w) -((\nu_E - w) \cdot  w) w|^{2+\g}
+f(w) +\nabla f(w)\cdot (\nu_E - w).$$
Therefore, since $\nabla f(w) = \frac{x}{f_*(x)} $ and $f(w) = \nabla f(w) \cdot w$,
\begin{align*}
G(\nu_E)  \geq \frac{\lambda}{2^{2+\g}} \left|(\nu_E - w ) - ((\nu_E - w)\cdot w) w\right|^{2+\g}&  = \frac{\lambda}{2^{2+\g}}
(1-(\nu_E\cdot  w)^2)^{(2+\g)/2}\\
&= \frac{\lambda}{2^{2+\g}}((1-\nu_E\cdot  w)(1+\nu_E\cdot  w))^{(2+\g)/2}.
\end{align*}
We want to show that there exists some $c_1$ such that
\begin{equation}\label{lastone}G(\nu_E) \geq c_1 (1 - \nu_E \cdot  w)^{(2+\g)/2}.\end{equation}
When $ w \cdot \nu_E \geq -c_0$ for some fixed $0<c_0<1,$ then $G(\nu_E) \geq  \frac{\lambda}{2^{2+\g}}(1-c_0)^{(2+\g)/2} (1-\nu_E\cdot  w)^{(2+\g)/2}$ and \eqref{lastone} holds.
On the other hand, when $ w \cdot \nu_E < -c_0$ for $c_0$ small, we expect that $\frac{x}{f_*(x)}\cdot \nu_E$ must also be small and so $G(\nu_E)$ is not too small. 
Indeed,
$$m_{\PP} \leq f( w)  =\frac{x}{f_*(x)} \cdot  w  = \frac{|x|}{f_*(x)} \cos (\theta_1)\leq M_{\PP} \cos( \theta_1),$$
where $\theta_1$ is the angle between $ w$ and $\frac{x}{f_*(x)}.$ Similarly,
$$-c_0 \geq \nu_E \cdot  w = \cos( \theta_2),$$
where $\theta_1$ is the angle between $ w$ and $\nu_E$. Noting that $0<m_{\PP}/M_{\PP}<1$, so
 $\cos^{-1} (m_{\PP}/M_{\PP} )\in (0, \pi/2), $
  we let $\theta_0 =2 \cos^{-1 }(m_{\PP}/M_{\PP}) + \e,$ where $\e>0$ is chosen small enough so that $\theta_0<\pi.$ Letting $c_0 = - \cos(\theta_0),$ we deduce that
$\theta_1 \leq \cos^{-1} (m_{\PP}/M_{\PP})$ and $\theta_2 \geq \theta_0.$
Then 
$$\frac{x}{f_*(x)} \cdot \nu_E \leq \frac{|x|}{f_*(x)} \cos (\theta_2 - \theta_1) \leq M_{\PP} \cos\left( \cos^{-1}(m_{\PP}/M_{\PP})+ \e\right) \leq m_{\PP} - M_{\PP} c_{\e},$$
for a constant $c_{\e}>0$. Since $f(\nu_E) \geq m_{\PP}$, we have $G(\nu_E)\geq  M_{\PP} c_{\e}$, implying \eqref{lastone} because $1 -\nu_E \cdot  w\leq 2$.

H\"{o}lder's inequality and \eqref{lastone} imply
\begin{align*}
 \int_{\partial^* E} 1 - \nu_E \cdot  w \ d\Hn&
\leq
\Hn(\partial^*E)^{\g/(2+\g)} \Big( \int_{\partial^* E} (1-\nu_E\cdot  w)^{(2+\g)/2}d\Hn\Big)^{2/(2+\g)} 
\\
&= 
c_1^{-2/(2+\g)} P(E)^{\g/(2+\g)}
 \Big( \int_{\partial^* E}c_1 (1-\nu_E\cdot  w)^{(2+\g)/2}d\Hn\Big)^{2/(2+\g)} \\
&\leq 
c_1^{-2/(2+\g)} P(E)^{\g/(2+\g)} \Big( \int_{\partial^* E} G(\nu_E)d\Hn\Big)^{2/(2+\g)}.
\end{align*}
Dividing by $n|K|^{1/n}|E|^{1/n'}$ and taking the square root, we obtain
$$\BPP^*(E) \leq c_1^{-1/(2+\g)} \Big(\frac{P(E)}{n|K|^{1/n}|E|^{1/n'}}\Big)^{\g/2(2+\g)} \beta_{\PP}(E)^{2/(2+\g)}.$$
\end{proof}

\bibliographystyle{plain}
\bibliography{references2}
\end{document}